\documentclass[10pt]{amsart}\usepackage{amsmath,amssymb,amsthm,graphicx,comment}

\usepackage{wasysym}
\usepackage{amsmath}
\usepackage{amssymb}
\usepackage{amsthm}
\usepackage{tikz}
\usetikzlibrary{matrix,arrows,backgrounds,shapes.misc,shapes.geometric,patterns,calc,positioning}
\usetikzlibrary{calc,shapes}
\usepackage{wrapfig}
\usepackage{epsfig}  
\usepackage{amsmath}
\usepackage{amsfonts}
\usepackage{amssymb}
\usepackage{graphicx}
\usepackage{multicol}
\usepackage[all]{xy}
\usepackage{xcolor}
\usepackage{color}
\usepackage{float}

\usepackage{graphics}

\usepackage[left=2.3cm,right=2.3cm,top=2.5cm,bottom=2.2cm]{geometry}

\newcommand{\CMP}{{\textup{CMP}}}

\newcommand{\CMI}{{\textup{CMI}}}

\newcommand{\B}{{\mathrm{B}}}
\newcommand{\D}{{\sf D }}
\newcommand{\Hom}{{\textup{Hom}}}
\newcommand{\End}{{\textup{End}}}
\newcommand{\Ext}{{\textup{Ext}}}

\renewcommand{\mod}{{\textup{mod}}}

\newcommand{\rad}{{\textup{rad}}}

\newcommand{\add}{{\textup{add}}}

\newcommand{\injdim}{{\textup{inj.dim}}}
\newcommand{\projdim}{{\textup{proj.dim}}}
\newcommand{\ind}{\textup{ind}}

\newcommand{\raw}{\rightarrow}

\newcommand{\law}{\leftarrow}

\newcommand{\Cc}{\mathcal{C}}
\newcommand{\Dd}{\mathcal{D}}
\newcommand{\Ee}{\mathcal{E}}

\newcommand{\Kk}{\mathcal{K}}
\newcommand{\Tt}{\mathcal{T}}
\newcommand{\Mm}{\mathfrak{M}}
\newcommand{\Nn}{\mathcal{N}}

\newcommand{\Ss}{\mathcal{S}}

\newcommand{\Xx}{\mathfrak{X}}

\newtheorem{teo}{Theorem}[section]
\newtheorem{lema}[teo]{Lemma}

\newtheorem{coro}[teo]{Corollary}

\newtheorem{prop}[teo]{Proposition}

\newtheorem{ex}[teo]{Example}
\newtheorem{rema}[teo]{Remark}

\newtheorem{df}[teo]{Definition}

\newtheorem{theorem*}{Theorem}
\newtheorem{cor*}{Corollary}

\usepackage{color}

\AtEndDocument{\bigskip{\footnotesize%

  \textsc{Departamento de Matem\'atica, Universidad Nacional de Mar del Plata, Dean Funes 3350, Mar del Plata, Argentina} \par
  \textit{E-mail address} \texttt{elsener@mdp.edu.ar} \par
  \textsc{Department of Mathematics, University of Connecticut, Storrs, CT 06269-3009, USA} \par
  \textit{E-mail address} \texttt{schiffler@math.uconn.edu}

}}

\begin{document}\date{\today}
%%
%% The title of the paper goes here.  Edit to your title.
%%
\title{On syzygies over 2-Calabi-Yau tilted algebras}
%%
%% Now edit the following to give your name and address:
%% 
\author{Ana Garcia Elsener \and Ralf Schiffler}

\thanks{ The first author was supported by CONICET, PICT 2013-0799 ANPCyT. Both authors were supported by the NSF grant DMS-1254567, by the Universidad Nacional de Mar del Plata, and by the University of Connecticut.}

\begin{abstract} 
We characterize the syzygies and co-syzygies over 2-Calabi-Yau tilted algebras in terms of the Auslander-Reiten translation and the syzygy functor. We explore connections between the category of syzygies, the category of Cohen-Macaulay modules, the representation dimension of algebras and the Igusa-Todorov functions. In particular, we prove that the Igusa-Todorov dimensions of $d$-Gorenstein algebras are equal to $d.$

For   cluster-tilted algebras of Dynkin type $\mathbb{D}$, we give a geometric description of the stable Cohen-Macaulay category in terms of tagged arcs in the punctured disc.
 We also describe  the action of the syzygy functor in a  geometric way. This   description allows us to compute the Auslander-Reiten quiver of the stable Cohen-Macaulay category   using tagged arcs and geometric moves.
\end{abstract}

\maketitle

\section{Introduction} The concept of a 2-Calabi-Yau tilted algebra is  a natural generalization of the concept of a cluster-tilted algebra. 
A cluster-tilted algebra is the endomorphism algebra of a cluster-tilting object in the cluster category of a hereditary algebra. The 2-Calabi-Yau tilted algebras are obtained by replacing the cluster category by a 2-Calabi-Yau triangulated category.
Examples of 2-Calabi-Yau tilted algebras that are not cluster-tilted are the Jacobian algebras of the quivers with potential arising from a cluster algebra of non-acyclic type. In particular, this includes the Jacobian algebras from triangulated surfaces other than the disc with 0,1, or 2 punctures and the annulus without punctures.
 Cluster  categories and cluster-tilted algebras were introduced in \cite{BMRRT,CCS,BMR}. The generalization to 2-Calabi-Yau categories was given in \cite{KR}. Cluster algebras were defined in \cite{FZ} and their quivers with potentials in \cite{DWZ}. For triangulated surfaces, the Jacobian algebras were introduced in \cite{ABCP,L}.

From now on, let $\B$ be a 2-Calabi-Yau tilted algebra.  We study the stable category $\underline{\CMP}(\B)$ of Cohen-Macaulay modules over $\B$. A $\B$-module $M$ is (projectively) Cohen-Macaulay if $\Ext^i_{\B}(M,\B)=0$, for all $i>0$.
It was proved in \cite{KR} that $\B$ is  $d$-Gorenstein, with $d=0$ or $1$, and that the stable category of Cohen-Macaulay modules is 3-Calabi-Yau. 
 We denote by $\ind\,\underline{\CMP}(\B)$ the set of  indecomposable objects in  $\underline{\CMP}(\B)$.
 Let $\Omega $ be the syzygy functor and $\tau$ the Auslander-Reiten translation. Our first main result gives the following characterization.

\begin{theorem*}   Let $M$ be an indecomposable $\B$-module. Then the following statements are equivalent.
\begin{itemize}
\item[(1)] $M$ is a non-projective syzygy,
\item[(2)] $M\in \ind\, \underline{\CMP}(\B)$,
\item[(3)] $\Omega^{2}\tau M \simeq M$,
\item[(4)] $M$ is non-projective and $\Omega^{-2} M \simeq \tau M$.
\end{itemize}
\end{theorem*}

We obtain the following corollary  on selfinjective 2-Calabi-Yau tilted algebras.
\begin{cor*} 
 Suppose that $\B$ is selfinjective and such that the Nakayama functor has finite order $m$. Let $M$ be an indecomposable, non-projective $\B$-module. Then  $\tau^{2m} M\cong M$.
 In particular, if $\B$ is symmetric then $\tau^2 M\cong M$.
\end{cor*}

The representation dimension of cluster-tilted algebras has   been studied in \cite{GT} for the special case of cluster concealed algebra, thus cluster-tilted algebras whose corresponding cluster-tilting object is preprojective. The authors of \cite{GT} have shown  that in this case the representation dimension  is at most 3. Cluster concealed algebras may or may not be tame and, on the other hand, tame cluster-tilted algebras need not be cluster concealed.
 Using results of \cite{BO} and \cite{R}, we extend the result to tame cluster-tilted algebras. 
\begin{cor*}
 The representation dimension of a tame cluster-tilted algebra is at most 3.
\end{cor*}

We also study the two Igusa-Todorov functions $\phi$ and $\psi$. These functions were introduced in \cite{IT} to study the relation between the representation dimension and the finitistic dimension of an algebra. 
The supremum of  $\phi$, or $\psi$ respectively,  is called the  Igusa-Todorov $\phi$-dimension, respectively $\psi$-dimension,  of the algebra.
 We obtain the following result in a more general setting, namely for $d$-Gorenstein algebras. This result has been proved independently in \cite{Ma}. 
\begin{theorem*}
 For any $d$-Gorenstein algebra, the Igusa-Todorov dimensions are both equal to $d$.
\end{theorem*}
Thus for 2-Calabi-Yau tilted algebras, both  Igusa-Todorov dimensions are at most 1.
\smallskip

In the second part of the paper, we give a geometric realization of the stable Cohen-Macaulay category  $\underline{\CMP}(\B)$  for all 2-Calabi-Yau tilted algebras from surfaces without punctures as well as for those from the disc with one puncture. The unpunctured case corresponds to the gentle algebras introduced in \cite{ABCP} and the case of  the punctured disc to the cluster-tilted algebras of type $\mathbb{D}$. While the unpunctured case is very simple, the case of the punctured disc is less so and requires a refinement of the combinatorial  model of the cluster category of \cite{S}.
 In \cite{CGL}, it was proved that in this case, $\underline{\CMP}(\B)$ is equivalent to a union of stable module categories of selfinjective cluster-tilted algebras.
We describe the indecomposable objects in $\underline{\CMP}(\B)$ in Theorem~\ref{teofull} as arcs in the punctured disc and we also describe the action of the syzygy functor $\Omega$ on  these arcs in Theorem~\ref{unlema}.
Furthermore, we construct the Auslander-Reiten quiver of  $\underline{\CMP}(\B)$ in terms of arcs and elementary moves.

The paper is organized as follows. In section \ref{sect 2}, we recall basic facts about Gorenstein algebras and prove our result on the Igusa-Todorov dimensions. Section~\ref{sect 3} is devoted to the study of 2-Calabi-Yau tilted algebras and contains Theorem 1 and its corollaries. The geometric realizations of the stable Cohen-Macaulay categories are given in section~\ref{sect 4} for the unpunctured case and in section \ref{sect 5} for the punctured disc.

\section{Gorenstein Artin algebras}\label{sect 2}
Let $R$ be a commutative artinian ring. In this section we consider $\Lambda$ an Artin $R$-algebra, and $\mod \Lambda$ the category of finitely generated right $\Lambda$-modules. Given a $\Lambda$-module $M$, denote the kernel of a projective cover by $\Omega M$. $\Omega M$ is called the syzygy of $M$. We let $\Omega^0 M=M$. The projective dimension of $M$, $\projdim M$, is the minimal length among all finite projective resolutions of $M$, and $M$ has infinite projective dimension if there is no finite projective resolution. Similarly, the injective dimension $\injdim M$, is the minimal length among all finite injective resolutions of $M$, and $M$ has infinite injective dimension if there is no finite injective resolution. We denote by $D$ the duality functor $\Hom_R (-,J)$ where $J$ is the injective envelope of the direct sum of the simple modules, \cite[II.3] {ARS}. When $\Lambda$ is a finite dimensional $k$-algebra over an algebraically closed field, there exists a quiver $Q$ such that $\Lambda$ is the quotient of the path algebra $k Q$ by an admissible ideal $I$.
In this case, we denote by $e_i$ the idempotent associated to the lazy path at the vertex $i$, and by $P(i) = e_i \Lambda$  the corresponding indecomposable projective module, by  $I(i)=D(\Lambda e_i)$ the   indecomposable injective module and  by $S(i)$ the simple top of $P(i)$, see for example \cite[Chapter III]{ASS} or \cite[Chapter 2]{S14}.

\begin{df} An Artin algebra $\Lambda$ is said to be \emph{Gorenstein of dimension $d$} (d-Gorenstein) if $\projdim D(\Lambda)=\injdim \Lambda =d < \infty $.
\end{df}

\begin{df} A $\Lambda$-module $M$ is said to be \emph{projectively Cohen-Macaulay} if $\Ext_\Lambda^{i}(M,\Lambda)=0$ for all $i > 0$. A $\Lambda$-module $N$ is \emph{injectively Cohen-Macaulay} if $\Ext_\Lambda^{i}(D\Lambda,N)=0$ for all $i > 0$.   
\end{df}

Denote by $\CMP (\Lambda)$ the full subcategory of $\textup{mod}\,\Lambda$ of projectively Cohen-Macaulay modules. The category $\CMP(\Lambda)$ is an exact subcategory, it is Frobenius and its projective-injective objects are precisely the projective modules in $\mod \Lambda$. Correspondingly, the category of injectively Cohen-Macaulay modules $\CMI (\Lambda)$ is Frobenius and its projective-injective objects are the injective modules in $\mod \Lambda$. 
The stable category $\underline{\CMP}(\Lambda)$ is a triangulated category whose inverse shift is given by the usual syzygy operator in $\mod \Lambda$, $\Omega_\Lambda$. Dually, the category $\underline{\CMI}(\Lambda)$  is a triangulated category whose shift is given by the usual co-syzygy operator in $\mod \Lambda$,  $\Omega^{-1}_\Lambda$. 
 There are triangle equivalences between $\underline{\CMP}(\Lambda)$, $\underline{\CMI}(\Lambda)$ and the singularity category given by the  localization $\Dd^{b}(\Lambda)\Ss^{-1}$, where $\Ss$ is the set of morphisms whose cone lies in $\Kk^{b}(proj \Lambda)$. See \cite[Section 4.8]{Bu}. 
\begin{rema}\label{rema11} The Auslander-Reiten (AR) translations $\tau,\tau^{-1}$ in $\mod\Lambda$ induce  quasi-inverse triangle equivalences $\tau: \underline{\CMP} (\Lambda) \raw \underline{\CMI}(\Lambda) $ and $\tau^{-1}: \underline{\CMI}(\Lambda) \raw \underline{\CMP}(\Lambda) $, see \cite[Chapter X]{BeR}.
\end{rema}

\subsection{Igusa-Todorov functions}

The Igusa-Todorov functions, $\phi$ and $\psi$, were introduced in order to study the relation between the representation dimension and the finitistic dimension in the context of Artin algebras. In a sense, they generalize the concept of projective dimension, meaning that for a module $M$ of finite projective dimension we have $\projdim M = \phi(M)=\psi(M)$. The important change occurs when the projective dimension of a module is not finite. We only give the definitions of Igusa-Todorov functions. For further details, see \cite{IT}. Let $K_0$ be the abelian group generated by the symbols $[M]$ for each isomorphism class $M$ in $\mod \Lambda$, modulo the relations $[M \oplus N]=[M]+[N]$ and $[P]=0$ for $P$ projective. Define the linear morphism $L:K_0 \raw K_0$ via $L[M]=[\Omega M]$, where $\Omega$ is the syzygy operator. For $M$ in $\mod \Lambda$, consider the direct sum $M = \bigoplus\limits_{i=1}^{k} M_i$, where all $M_i$ are indecomposable. Denote by $L^{0}(M)=\langle[M_1],\ldots, [M_k]\rangle$ the  subgroup of $K_0$ generated by the symbols $[M_1],\ldots, [M_k]$, then denote by $L^{1}(M)=L(L^{0}[M])$, and  so on.
\begin{df} The \emph{first Igusa-Todorov function} is defined by $\phi(M)=t$, where $t$ is the smallest integer such that $L:L^{t+s}(M)\raw L^{t+s+1}(M)$ is an isomorphism for all $s\geqslant 0$.
\end{df}
\begin{df} The \emph{second Igusa-Todorov function} is given by $\psi(M)=\phi(M)+k$, where $k$ is the largest finite projective dimension in a summand of the module $\Omega^{\phi(M)}M$.
\end{df}
\begin{ex} Let $\Lambda$ be the algebra given by $kQ/rad^{2}$, where $Q$ is the quiver 
\[\xymatrix{1\ar@(l,d)[]\ar[r] & 2\ar[r] & 3}\]
\\
We compute $\phi(M)$ and $\psi(M)$ for the module $M=I(1)\oplus S(1)$. We have the following projective resolutions:
\\
\[\xymatrix@C9pt@R9pt{0\ar[rr] && 3\ar[rr]\ar@{->>}[rd] && {\begin{array}{c}
2 \\ 
3
\end{array} }\ar[rr]\ar@{->>}[rd] &&{\begin{array}{c}
1 \\ 
1 2
\end{array}}\ar[rr]&&I(1) \ar[rr]&&0\\
&& 
&3\ar@{^{(}->}[ru] &&2\ar@{^{(}->}[ru]}\]
$  $
\[\xymatrix@C8pt@R8pt{ \dots \ar[rr]&&
{\begin{array}{c}
1 \\ 
1 2
\end{array}} \oplus {\begin{array}{c}
2 \\ 
3 \\
\end{array} }\oplus 3\ar[rr]\ar@{->>}[rd]&&{\begin{array}{c}
1 \\ 
1 2
\end{array}} \oplus {\begin{array}{c}
2 \\ 
3
\end{array} }\ar[rr]\ar@{->>}[rd] &&{\begin{array}{c}
1 \\ 
1 2
\end{array}}\ar[rr]&&S(1) \ar[rr]&&0\\
&1 \oplus 2 \oplus 3 \ar@{^{(}->}[ru] && 1 \oplus 2 \oplus 3\ar@{^{(}->}[ru]&&1\oplus 2\ar@{^{(}->}[ru]}\]
Note that $[S(3)]=0$, since $S(3)$ is projective. The sequence of abelian groups is given by:

\begin{equation*}
L^{0}(M)=\langle [I(1)],[S(1)]\rangle \xrightarrow{L}\langle [S(2)],[S(1)]+[S(2)] \rangle \xrightarrow{L} \langle [S(1)]+[S(2)]\rangle \xrightarrow{L} \langle [S(1)]+[S(2)]\rangle \raw \cdots
\end{equation*}
Starting from the second syzygy, the rank of the abelian groups is one. So, $\phi(M)=2$. And $\psi (M)=3$, because $\projdim S(2)= 1$.
\end{ex}
The $\phi$-dimension and $\psi$-dimension of an algebra $\Lambda$ are given by $\phi\dim(\Lambda)= \sup \lbrace \phi(M) :  M\in \mod \Lambda \rbrace$ and $\psi\dim(\Lambda)= \sup \lbrace \psi(\Lambda) :  M\in \mod \Lambda \rbrace$. 
\\\\
Now consider $\Lambda$ a $d$-Gorenstein Artin algebra. As stated before, $\Omega_\Lambda$ is the inverse shift in the stable category $\underline{\CMP}(\Lambda)$, so it is a  bijective map on the non-projective $d$-th syzygies in $\mod \Lambda$. We deduce the following.

\begin{teo}\label{teoH} Let $\Lambda$ be a $d$-Gorenstein Artin algebra, then $\phi$dim$(\Lambda)=\psi$dim$(\Lambda )=d$.
\end{teo} 
\begin{rema}
 This result has been obtained independently in \cite{Ma}. 
\end{rema}
\begin{proof} Let $\Lambda$ be as in the hypothesis and let $M$ be a $\Lambda$-module written as a direct sum of indecomposables $M=\oplus_{i=1}^r M_i$. Consider a projective resolution
\begin{equation*}
0 \raw \Omega^d M \raw P_{d-1} \raw \cdots \raw P_0 \raw M \raw 0.
\end{equation*}
It is known from homological algebra that $\Ext^{d+j}(M,\Lambda) \simeq \Ext^j(\Omega^d M, \Lambda)$ for $j > 0$. By hypothesis, $\injdim \Lambda =d$, then for all $j>0$ we have $\Ext^{d+j}(-, \Lambda)=0$. Therefore, $\Omega^d M$, and all its summands, belong to $\CMP (\Lambda)$. Since $\Omega$ acts as the inverse shift on the triangulated category $\underline{\CMP}(\Lambda)$, the morphism $L$ is injective over the subgroup $\Ss = \langle\ind\,\underline{\CMP}(\Lambda) \rangle$ of $K_0$. For $j \geqslant d$, $S_j =\langle [\Omega^j M_1],\ldots,[\Omega^j M_r]\rangle $ is a subgroup of $\Ss$ of finite rank. The morphism $L$ is such that $L(S_j) \cong S_{j+1}$, and, since it is injective over $S_j$ for all $j \geqslant d$, it preserves the rank from $d$ on. According to the definition, $\phi(M) \leqslant d$. Also, $\phi (D \Lambda)=\projdim D \Lambda =d$, then $\phi$dim$(\Lambda)=d$. Now, we analyze the $\psi$dim. Let $M$ be a $\Lambda$-module and let $\phi(M)=t \leqslant d$. By definition $\psi(M)=\phi(M)+ \projdim Z$, where $Z$ is a summand of $\Omega^t M$ with finite projective dimension. Then $N=\Omega^{d-t}Z$ is a summand of $\Omega^d M$. We have that $N$ belongs to $\CMP$.
 If $N$ is projective, then $\projdim Z \leqslant d-t$ and this implies $\psi(M) \leqslant d$. If $N$ is not projective, thus $N$ is nonzero in $\underline{\CMP}$, then $N$ has infinite projective dimension, which is impossible because we are assuming that $\projdim Z$ is finite. Therefore, $\psi(M) \leqslant d$ and, since $\psi(D \Lambda)= d$, this yields $\psi$dim$\Lambda=d$.
\end{proof} 

\begin{rema} In \cite{HL} it is proved that an Artin algebra is selfinjective if and only if the $\phi$ and $\psi$ dimension are zero. For $d >0$, the converse of Theorem \ref{teoH} is not true, as we see in the following example  due to Mata \cite{Ma}.
\end{rema}

\begin{ex} Let $Q$ be the following quiver, and $A=kQ/rad^2$.
\[\xymatrix{1\ar[r] & 2\ar[r] & 3\ar[d]\\
& 5\ar[u]& 4\ar[l]}\] 
This algebra is not 1-Gorenstein because $\projdim I(2)= \infty$. The only indecomposable non-projective modules are $I(2)$ and the simple modules $S(j)$ for $j \neq 1$. Among these modules, the indecomposable first syzygies are the simple modules and they have infinite projective dimension. The morphism $L$ permutes the modules $\{ S(j) : j \neq 1 \}$. Then, $\phi$dim$(A)= \psi$dim$(A)=1$. 
\end{ex}

\section{2-Calabi-Yau tilted algebras}\label{sect 3}
The main result of this section is the characterization of the modules in the categories $\underline{\CMP}(\B)$ and $\underline{\CMI}(\B)$ for a 2-Calabi-Yau tilted algebra $\B$ in terms of the functors $\tau$, $\tau^{-1}$, $\Omega$ and $\Omega^{-1}$. As an immediate corollary, we obtain that $\Omega(\mod \B)= \CMP(\B)$. First, let us recall definitions and results about 2-Calabi-Yau categories and 2-Calabi-Yau tilted algebras. Let $\Cc$ be a $k$-linear triangulated category. For $X,Y$ in $\Cc$ we have $\Ext^{n}(X,Y)=\Hom(X,Y[n])$. 

\begin{df} Let $k$ be an algebraically closed field. A $k$-linear triangulated category $\Cc$ with split idempotents, suspension functor $[1]$ and finite dimensional $\Hom$-spaces is called \emph{2-Calabi-Yau} if
\begin{equation*}
D\Ext_\Cc(X,Y) \simeq \Ext_\Cc(Y,X).
\end{equation*}
\end{df}
We will often write 2-CY instead of 2-Calabi-Yau. Examples of 2-CY categories are the cluster categories defined in \cite{BMRRT}, \cite{CCS} and \cite{Am}. For more information about 2-CY categories we refer to \cite{Am11}. 
\begin{df} Let $\Cc$ be a 2-CY category. An object $T \in \Cc$ is called \emph{cluster-tilting} if $T$ is basic and $\add T= \lbrace X \in \Cc : \Ext_\Cc (X,T)=0\rbrace$.
\end{df}

\begin{df} The endomorphism algebra of a cluster-tilting object $\End_\Cc(T)$ is called a \emph{2-Calabi-Yau tilted algebra}.
\end{df}
It was shown in \cite{BMR} that there is an equivalence of categories
\begin{equation}\label{bmr}
F: \Cc/\add T [1] \raw \mod \End_\Cc(T)
\end{equation}
given by $X \mapsto \Hom_\Cc(T,X)$. This equivalence maps $\add T$ to the projective $\End_\Cc(T)$-modules and $\add T[2]$ to the injective $\End_\Cc(T)$-modules. 
\\
It has been proved in \cite{KR} that every 2-CY tilted algebra is Gorenstein of dimension at most one. Consequently, the projective dimension of a module can only be 0,1 or infinity. And the same holds for the injective dimension.

\begin{lema}\label{lema1} Let $\B$ be a 2-CY tilted algebra and let $M\in \mod \B$. Let $P_1\raw P_0\xrightarrow{h} M\raw 0$ be a minimal projective presentation. Then, there is an exact sequence
\begin{equation*}
\tau^{-1}P_0 \xrightarrow{f} \tau^{-1}M \xrightarrow{g} P_1 \raw P_0 \xrightarrow{h} M \raw 0.
\end{equation*}
The induced epimorphism $\overline{g}:\tau^{-1}M \raw \Omega^{2}M$ is an isomorphism if and only if $f=0$.
\end{lema}
\begin{proof} Consider an object $X$, without summands in $\add T [1]$, in $\Cc$ such that $\Hom_\Cc(T,X)=M$. It has been shown in \cite{KR} that there is a triangle in $\Cc$
\begin{equation*}
T_1 \raw T_0 \xrightarrow{\tilde{h}} X \raw T_1[1]
\end{equation*}
where $T_0,T_1\in\add T$, and $\tilde{h}$ is a minimal $\add T$-approximation. Applying $\Hom_\Cc(T, - )$, we obtain 
\begin{equation*}
\Hom_\Cc(T,T_0[-1]) \raw \Hom_\Cc(T,X[-1])\raw \Hom_\Cc(T,T_1)\raw \Hom_\Cc(T,T_0)\xrightarrow{h} \Hom_\Cc(T,X)\raw 0
\end{equation*}
which gives the following exact sequence in $\mod \B$
\begin{equation*}
\tau^{-1}P_0 \xrightarrow{f} \tau^{-1}M \xrightarrow{g} P_1 \raw P_0 \raw M \raw 0,
\end{equation*}
where $P_1 \raw P_0 \raw M \raw 0$ is a minimal projective presentation and every minimal projective presentation arises this way. By the exactness of the sequence, the induced epimorphism $\overline{g}:\tau^{-1}M \raw \Omega^{2} M$ is an isomorphism if and only if $f=0$. 
\end{proof}

We are now ready to state our first main result.  We denote by $\ind\,\underline{\CMP}(\B)$ the set of indecomposable objects in $\underline{\CMP}(\B)$. Abusing notation,  we will denote the class of a $B$-module $M$ in the category $\underline{\CMP}(\B)$ also by $M$.

\begin{teo}\label{teo} Let $\B$ be 2-CY tilted algebra. Let $M$ and $N$ be indecomposable modules in $\mod \B$. Then the statements \textup{(a1)-(a4)} (respectively \textup{(b1)-(b4)}) are equivalent
\begin{multicols}{2}
\begin{itemize}
\item[(a1)] $M$ is a non-projective syzygy,
\item[(a2)] $M\in\ind\,\underline{\CMP}(\B)$,
\item[(a3)] $\Omega^{2}\tau M \simeq M$,
\item[(a4)] $M$ is non-projective and $\Omega^{-2} M \simeq \tau M$.
\end{itemize}
\columnbreak
\begin{itemize}
\item[(b1)] $N$ is a non-injective co-syzygy,
\item[(b2)] $N\in\ind\,\underline{\CMI}(\B)$,
\item[(b3)] $\Omega^{-2}\tau^{-1} N \simeq N$,
\item[(b4)] $N$ is non-injective and $\Omega^{2} N \simeq \tau^{-1} N$.
\end{itemize}
\end{multicols}
\end{teo}
\begin{proof}
(a1)$\Rightarrow$(a2).

Let $M$ be an indecomposable non-projective syzygy. Then there is a short exact sequence $0\raw M \raw P_0 \raw N \raw 0$ in $\mod \B$. Applying $\Hom_\B(- ,\B)$ yields an exact sequence
\begin{equation*}
\cdots\raw \Ext^1_\B(P_0,\B)\raw \Ext^1_\B(M,\B) \raw \Ext_\B^{2}(N,\B)\raw \cdots .
\end{equation*}
The first term is zero because $P_0$ is projective and the last term is zero because $\injdim \B \leqslant 1$, by the Gorenstein condition. Then $\Ext^1_\B(M,\B)=0$, so $M$ is in $\CMP (\B)$, and, since $M$ is not projective, then $M\in\ind\,\underline{\CMP} (\B)$.

(a2)$\Rightarrow$(a3).
Let $M\in\ind\,\underline{\CMP}(\B)$. Since $M$ is not projective, $\tau M \neq 0$. Consider the exact sequence given in Lemma \ref{lema1}, but for the module $\tau M$
\begin{equation*}
\tau^{-1}P_0\xrightarrow{f}M \xrightarrow{g}P_1 \raw P_0 \raw \tau M \raw 0.
\end{equation*}
So the epimorphism $\overline{g}:M \raw \Omega^{2} \tau M$ is an isomorphism if and only if $f=0$. The morphism $f$ belongs to $\Hom_\B (\tau^{-1}P_0,M)$. From the Auslander-Reiten formulas, and the fact that $\injdim P_0\leqslant 1 $, we have
\begin{equation}\label{eq1}
 \D \Ext^1_\B(M,P_0) \simeq \underline{\Hom}_\B (\tau^{-1}P_0,M) \simeq  \Hom_\B(\tau^{-1}P_0,M).
\end{equation}
But since $M \in \underline{\CMP}(\B)$, we have $\Ext^1_\B(M,P_0)=0$, and thus $f$ must be zero and $M \simeq \Omega^{2}\tau M$.

(a3)$\Rightarrow$(a1).
If $M\simeq \Omega^{2}\tau M$, then $M$ cannot be projective because $\tau M \neq 0$, and $M$ is a syzygy.

(a2)$\Rightarrow$(a4).
First notice that the proof of the equivalence (b1)$\Leftrightarrow $(b2)$\Leftrightarrow $(b3) is similar to (a1)$\Leftrightarrow $(a2)$ \Leftrightarrow $(a3), using the dual of Lemma~\ref{lema1}.  
If $M\in\ind\,\underline{\CMP}(\B)$, then by Remark 2.3, $\tau M\in\ind\,\underline{\CMI}(\B)$. Using (b2)$\Leftrightarrow$(b3), we have $\Omega^{-2}M \simeq(\Omega^{-2}\tau^{-1})\tau M \simeq \tau M$.

(a4)$\Rightarrow$(a2).
If $\tau M \simeq \Omega^{-2}M$ and $M$ is not projective, then $\tau M$ is a co-syzygy and is not injective. Hence, $\tau M\in\ind\,\underline{\CMI}(\B)$ and, by Remark 2.3,  $M\in\ind\,\underline{\CMP}(\B)$.  The proof of  equivalence (b4) and the other statements of part (b) is similar. 
\end{proof}

The equivalence of (a1) and (a2) shows that $\CMP (\B)= \Omega(\mod \B)$,  and the set $\ind\,\underline{\CMP}(\B)$ of indecomposable objects in the stable category can be identified with the set of non-projective indecomposable syzygies in $\mod\,\B$. This holds for all 1-Gorenstein Artin algebras. It is known that, for a $d$-Gorenstein algebra, $M$ is projectively Cohen-Macaulay if and only if $M$ is a $d$-th syzygy, see \cite[Proposition 6.20]{Be00}. Each module over a $d$-Gorenstein algebra either has infinite projective dimension or has projective dimension at most $d$. In this context, as explained in \cite{Be00}, the modules called projectively Cohen-Macaulay in these notes are the modules of Gorenstein dimension zero in \cite[Proposition 3.8]{ABr}, and Gorenstein-projective modules in \cite[Chapter 10]{EJ}.

Denote by $\Nn=\D\Hom_\Lambda(-,\Lambda)\colon\mod \Lambda\to\mod\Lambda$ the Nakayama functor. If $\Lambda$ is a selfinjective $k$-algebra then the functors $\Omega^{-2} \tau$ and $\Nn$ induce naturally isomorphic functors  $\underline{\mod}\Lambda\to \underline{\mod}\Lambda$ on the stable category. We say that the Nakayama functor has \emph{finite order} if there exists $m>0$ such that the functor $\Nn^m$ is naturally isomorphic to the identity.
  
We get the following:
\begin{coro}\label{cor 3.6} If a 2-CY tilted algebra $\B$ is selfinjective and the Nakayama functor $\mathcal{N}_\B$ has finite order $m$, then $\tau^{2m} M \simeq M$ for all $M$ in $\underline{\mod} \B$.  In particular, if $\B$ is symmetric then $\tau^2 M\cong M$.
\end{coro}
\begin{proof}
Since $\B$ is selfinjective, we have $\underline{\CMP}(\B)= \underline{\mod}(\B)$. Therefore Theorem \ref{teo} (a4) implies that $\Omega^{-2}_\B M=\tau_\B M$ for all $M\in \underline{\mod}\B$. 
Now if $\mathcal{N}$ has order $m$, we get $M=\mathcal{N}^{m}M=(\Omega^{-2} \tau)^{m} M = \tau^{2m} M $.  The last statement of the corollary follows because for symmetric algebras the Nakayama functor is the identity.
\end{proof}

\begin{rema}  In the symmetric case Corollary \ref{cor 3.6} was already obtained in \cite{VD} and \cite{Lad}. Examples of symmetric 2-CY tilted algebras are  the Jacobian algebras arising from triangulations of closed surfaces.
\end{rema}

We have seen that for 2-CY tilted algebras, the indecomposable syzygies are the same as the indecomposable $\CMP$-modules. In \cite{R}, Artin algebras with a finite number of syzygies are called torsionless-finite. In this work, the author proves that for torsionless-finite Artin algebras the representation dimension is at most three. On the other hand, in \cite{BO} the authors prove that when $\B$ is a tame cluster-tilted algebra, $\CMP (\B)$ has finitely many indecomposable modules.
We thus obtain the following result.

\begin{coro} If $\B$ is a $\CMP$ finite 1-Gorenstein algebra, then its representation dimension is at most three. In particular, if $\B$ is a tame cluster-tilted algebra, then its representation dimension is at most three. 
\end{coro}

\begin{rema} The representation dimension of cluster-concealed algebras, a class of cluster-tilted algebra which overlaps with the tame ones, was studied in \cite{GT}. The authors compute the representation dimension by constructing an explicit Auslander generator. 
\end{rema}

\section{Geometric description of $\CMP$ modules: The unpunctured case}\label{sect 4}

An important class of 2-CY tilted algebras are those arising from surface triangulations. In this case, the 2-CY category is the generalized cluster category introduced in \cite{Am}. It is described in terms of arcs and elementary moves in \cite{CCS} for Dynkin type $\mathbb{A}$, in \cite{S} for Dynkin type $\mathbb{D}$, in \cite{BZ} for surfaces without punctures and in \cite{QZ} for surfaces with punctures. Starting with a marked surface and a triangulation, one gets the bound quiver of the 2-CY tilted algebra as adjacency quiver of the triangulation. The dimension vector of representations are given by the intersection number of curves with the arcs that form the triangulation, and the AR translation is given by elementary moves and change of tags in the surface. Motivated by these ideas, our goal is to find a geometric description of $\underline{\CMP}(\B)$ and of the $\Omega $-action in terms of curves. 
%\subsection{}\label{unpunctured}

In \cite{Ka}, the author computes $\CMP(A) $ where $A=kQ/I$ is a gentle algebra. The indecomposable modules in $\underline{\CMP}(A)$ are given by the non-projective indecomposable summands of $\rad P(a)$, where $a$ is a vertex in a cycle $\alpha_1 \ldots \alpha_n $ such that $\alpha_t \alpha_{t+1}\in I$ for all $t$. In the following example, we refer to the 2-CY tilted algebras from unpunctured surfaces introduced in \cite{ABCP}. These algebras are gentle, so the $\CMP$ and $\CMI$ modules can be described easily.

\begin{ex}\label{ex1} Geometric description of $\underline{\CMP}(\B)$ for a 2-CY tilted algebra arising from a triangulated surface without punctures. 
\\
 We have $M\in\ind\,\underline{\CMP}(\B)$  if and only if $M$ is a summand of $\rad P(j)$, where $j$ is a vertex lying in a 3-cycle $i \xrightarrow{\alpha_i} j \xrightarrow{\alpha_j} k \xrightarrow{\alpha_k} i$. %with $\rad^{2}=0$. 
Such 3-cycles are in bijection with internal triangles in the surface triangulation. We write the projective modules as string modules 

\begin{align*}
  P(i)&: \beta_i^{-1} \leftarrow  i  \ \raw \  j  \ \raw  \ \beta_j\\
  P(j)&: \beta_j^{-1} \leftarrow  j  \ \raw \  k  \ \raw  \ \beta_k\\
  P(k)&: \beta_k^{-1} \leftarrow  k  \ \raw \  i  \ \raw  \ \beta_i
\end{align*}
where $\beta_i,\beta_j$ and $\beta_k$ are paths in the quiver. We also use the notation $M(i \raw \beta_i)$ for the module associated to the string $i\ \raw \beta_i$. An internal triangle formed by arcs $i,j,k$ in the surface induces a periodic projective resolution:

\[\xymatrix@C8pt@R8pt{ \dots \ar[rr]&&
 P(i)\ar[rr]\ar@{->>}[rd]&& P(k) \ar[rr]\ar@{->>}[rd] && P(j)
\ar[rr]&& M(j \raw \beta_j) \ar[rr]&&0\\
&M(j \raw \beta_j) \ar@{^{(}->}[ru] && M(i \raw \beta_i) \ar@{^{(}->}[ru]&&M(k \raw \beta_k) \ar@{^{(}->}[ru]}\]
All possible indecomposable $\underline{\CMP}$ modules appear in that way. For a triangulated unpunctured surface, each internal triangle gives a periodic projective resolution with period 3 and uniserial modules as syzygies. The dimension vector for these modules is given by the crossing number in the red arcs in Figure \ref{exgentle}.  
\begin{figure}
\centering
\def\svgwidth{5.7in}
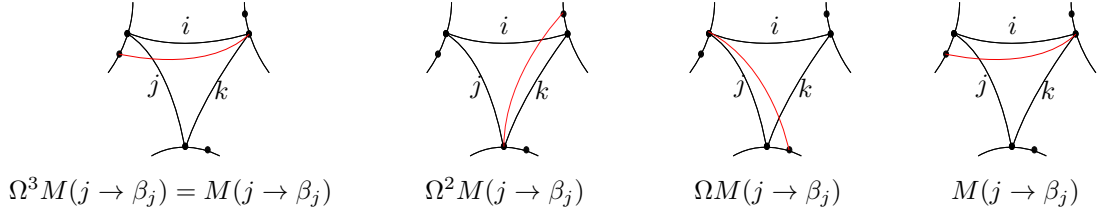
\caption{Arcs that correspond to indecomposable $\underline{\CMP}$ modules in a gentle algebra arising from a surface.}
\label{exgentle}
\end{figure}
The module $M(j \raw \beta_j)$ is associated to the arc $\gamma$ obtained from the arc $i$ by performing an elementary move counter-clockwise on one of the endpoints of $i$, in such way that the arc $\gamma$ crosses the arc $j$. By Remark \ref{rema11}, the modules in $\underline{\CMI}$ are associated to the AR translated of the arcs in Figure \ref{exgentle}, and can be obtained performing an elementary move in $\tau$ (clockwise) direction at the endpoints. 
\end{ex}
\begin{rema}
 An alternative way to prove this description of $\underline{\textup{CMP}}( \B)$ would be to use the formula of \cite{CSchr} for $\textup{Ext}^1$ over Jacobian algebras of unpunctured surfaces in terms of different types of crossings in the surface.  
\end{rema}
\section{Geometric description of CMP modules: The punctured disc} \label{sect 5}
In this subsection we describe $\underline{\CMP}(\B)$ for the disk with one puncture, or equivalently, for cluster-tilted algebras of Dynkin type $\mathbb{D}$. This gives an answer to \cite[ Remark 3.7]{Ka} in this case.
This situation is much more complicated than the unpunctured case, and we will need to distinguish between 3 types of triangulations.

The geometric definition of 
the cluster category of type $\mathbb{D}$ was given
in \cite{S}. The reader can also see  \cite[Section 3.3]{S14}. Since our surface now has a puncture, we have to work with tagged arcs. 
If $\alpha $ is an arc that is incident to the puncture, then it gives rise to two tagged arcs, one is tagged \emph{plain} and will still be denoted by $\alpha $ and the other is tagged \emph{notched} and will be denoted by $\alpha^{\Join}$. The underlying curve for both $\alpha$ and $\alpha^{\Join}$ is the same.
All possible quivers for a cluster-tilted algebra type $\mathbb{D}_n$ arise from triangulations of a punctured disc with $n$ marked points on the boundary. There is a correspondence between  tagged arcs and indecomposable objects in the cluster category, and this leads to a correspondence between  tagged arcs that are not in the triangulation and indecomposable $\B$-modules. If $\gamma$ is an arc, we denote by $X_\gamma$ the corresponding object in the cluster category. Let $M_\gamma$ be the indecomposable module corresponding to $\gamma$ if $\gamma$ is not in the triangulation, and let $M_\gamma =0$ otherwise. The AR translation in the cluster category is realized as follows, $\tau X_\gamma = X_\delta$ where $\delta$ is obtained from $\gamma$ by moving the endpoints to their clockwise neighbors along the boundary, and changing tag if $\gamma$ is incident to the puncture. The AR translation in $\mod \B$ is inherited from the AR translation in the cluster category. Let $a$ be an arc in the triangulation, the module associated to the arc $\tau^{-1}(a)$ is the indecomposable projective $P(a)$, and the module associated to the arc $\tau(a)$ is the indecomposable injective $I(a)$.
\\

We say that two arcs incident to the puncture are \emph{consecutive} if one of them is incident to a marked point $q$ and the other is incident to $\tau (q)$. In the following definition, we introduce 3 types of triangulations which are illustrated in Figure \ref{2 y 3}.

\begin{df}\label{types} We say that a triangulation on the punctured disc is of
\begin{enumerate}
\item \emph{type I} if there are three or more arcs, or two non-consecutive plain  arcs, incident to the puncture,
\item \emph{type II} if there are precisely two consecutive arcs incident to the puncture, or if there are two isotopic arcs incident to the puncture, one plain and one notched, and a third arc $\overline{a}$ as in Figure \ref{2 y 3}. 
\item \emph{type III} if there are two arcs incident to the puncture and the same marked point, and two arcs $\overline{c}$ and $\overline{d}$ as in Figure \ref{2 y 3}. 
\end{enumerate}
\end{df}

\begin{figure}
\centering
\def\svgwidth{5in}
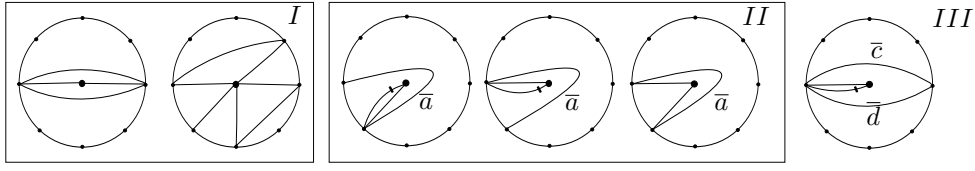
\caption{Partial triangulations of type I (framed left), of type II (framed center) and of type III (right).}
\label{2 y 3}
\end{figure}

All possible quivers for a cluster tilted algebra of type $\mathbb{D}$ arise from a triangulation of type I, II or III. Note however, that not all triangulations are of one of these types. In the following lemmas, we give a necessary condition for an object $X_\gamma$ to be associated to a module  $M_\gamma \in {\CMP}(\B)$. Our principal tool is the following consequence of the AR formulas.

\begin{rema}
 \label{rema util}
 Let $M$ be a $B$-module.
 \begin{itemize}
\item [\textup{(a)}] $M
\notin
 \CMP(\B)$ if and only if $\Hom_\B (\tau^{-1}P(t),M)
 \ne 0$ for some $t\in Q_0$. In particular, $\tau^{-1} P(t)\notin \CMP(\B).$
\item [\textup{(b)}] If $\projdim M =1$ then $  M\notin\CMP(\B)$.
\end{itemize}
\end{rema}
\begin{proof} Part
 (a) follows from the definition of $\CMP(\B)$ and equation (\ref{eq1}) in section \ref{sect 3}. Part (b) holds because  if $\projdim M = 1$, then $\Ext^1_\B(M,\B)\cong D\Hom_\B(\B,\tau M)\cong \tau M\ne 0$ since $M$ is not projective.
\end{proof}

Because of the equivalence (\ref{bmr}) of section \ref{sect 3}, to prove that $M_\gamma$ is not in $\CMP(\B)$ is enough to show that $\Hom_{\Cc}(\tau^{-1}T_t , X_\gamma)\neq 0$ and there is at least one non-zero morphism in this Hom-space that does not factor through $\add \tau T$. In the figures, we will often symbolize the object $X_a$, associated to an arc $a$ in the triangulation, by a circled $a$. For an arc $a$ in the triangulation, the associated object $X_a$ in the cluster category is equal to $\tau T_a = T_a [1]$, where $T_a$ is an indecomposable summand of the cluster tilting object $T$. Also, to say that the arc $\delta$ crosses the arc $a\in\Tt$ is equivalent to saying that $\Ext_\Cc(X_a, X_\delta) \neq 0$. Then, $\Hom_\Cc(T_a, X_\delta) \simeq \Hom_\Cc( X_a [-1], X_\delta) \simeq \Hom_\Cc (X_a, X_\delta [1]) \neq 0$. Thus, the objects $X_\delta$ such that the arc $\delta$ crosses the arc $a$ are precisely those in the support of $\Hom_\Cc (T_a, -)$.

 \subsection { Triangulations of type II and III and the set $\underline{\CMP}_\clubsuit$}  We first study these two types because they are easier. Type I is studied in the next subsection. 
 
 \begin{lema}\label{lema2y3} \begin{itemize}
\item [\textup{(a)}] Let $\mathcal{T}$ be a triangulation of type II and let $\overline{a}\in\mathcal{T}$ be as in Figure \ref{2 y 3}. Let $\gamma$ be an arc that crosses $\overline{a}$.
If $M_\gamma\in \ind\,\underline{\CMP}(\B)$ then, in the AR quiver, $M_\gamma$ lies at one of the positions marked with a $\diamond$ in Figure \ref{figura2}.
\item [\textup{(b)}] Let $\mathcal{T}$ be a triangulation of type III and let $\overline{c},\overline{d}\in\mathcal{T}$ be as in Figure \ref{2 y 3}. Let $\gamma$ be an arc that crosses $\overline{c}$ or $\overline{d}.$
If $M_\gamma\in \ind\,\underline{\CMP}(\B)$ then, in the AR quiver, $M_\gamma$ lies at one of the positions marked with a $\clubsuit$ or  $\ast$ in Figure \ref{figura3}.
\end{itemize}
\end{lema}

\begin{proof}  
(a) We prove the lemma for the right case in Figure 2. A similar analysis can be done for the other cases, resulting in the same relative positions for $M_\gamma$. The objects $X_\delta$ such that $\delta$ is in the support of $\Hom_\Cc(T_{\overline{a}},-)$ are: $\tau^{-1}T_{b}$, $T_{\overline{a}}$, $T_{\overline{c}}$, those in $\diamond$ position and those in the area $A$ in Figure 3. The objects in $A$ are in the support of $\Hom_\Cc(T_{\overline{a}},-)$ and of $\Hom_\Cc(\tau^{-1}T_{\overline{a}}, -)$. Then, for $X_\delta$ in $A$ there exist a non-zero morphism $f\in \Hom_\Cc(\tau^{-1}T_{\overline{a}},X_\delta)$ and, since $A$ does not contain summands of $\tau T$, $f$ does not factor through $\tau T$. This implies that $M_\delta$ is not in $\CMP(\B)$. The objects $T_{\overline{a}}$ and $T_{\overline{c}}$ are associated to  the projectives $P(\overline{a})$ and $P(c)$, so they are not in $\ind\,\underline{\CMP}(\B)$. The object $\tau^{-1} T_{b}$ is associated to $\tau^{-1}P(b)$, and by Remark \ref{rema util}, this module is not in ${\CMP}(\B)$. The only remaining positions are those labeled by a $\diamond$.

\begin{figure}
\centering
\def\svgwidth{4.9in}
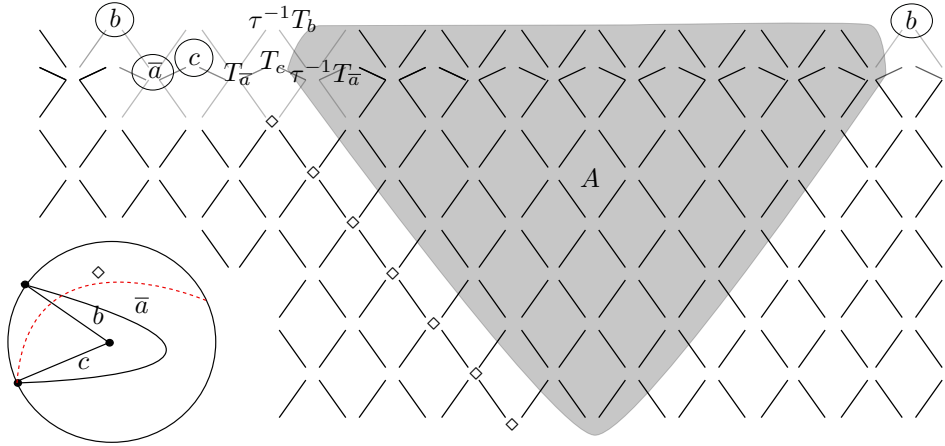
\caption{Lemma \ref{lema2y3} for a triangulation of type II. The position of $T_{\overline{a}}$, $T_{b}$ and $T_{c}$ make $\diamond$ the only admissible positions for arcs such that $\gamma $ crosses $\overline{a}$ and $M_\gamma$ can be in $\ind\,\underline{\CMP}$.}
\label{figura2}
\end{figure}

(b) The objects $X_\delta$ in $A_1$ are in the support of $\Hom_\Cc(T_{\overline{d}},-)$ and of $\Hom_\Cc(\tau^{-1}T_{\overline{d}},-)$, and by the same argument as in type II, we see that $M_\delta\notin \ind\,\underline{\CMP}(\B)$, for all $X_\delta$ in $A_1$.
\\
 The objects $X_\delta$ in $A_2$ are in the support of  $\Hom_\Cc(T_{a} \oplus T_{b}, -)$ and of $\Hom_\Cc(\tau^{-1}T_{a} \oplus \tau^{-1}T_{b}, -)$, and by the same argument as in type II, we see that
$M_\delta\notin\ind\,\underline{\CMP}(\B)$ for all $X_\delta $ in $A_2$.  
\\
 The area $A_3$ does not contain summands of $\tau T$. All $X_\delta$ in $A_3$ satisfy $\Hom_\Cc (\tau^{-1}T_{\overline{c}},X_\delta)\neq 0$. For the objects $X_\delta$ not labelled 2 in Figure \ref{figura3}, we can apply the same argument as before to prove that the associated $M_\delta$ are not in $\ind\,\underline{\CMP}(\B)$. Let $X$ be an indecomposable object whose position in the AR quiver is labelled 2 in Figure \ref{figura3}. This $X$ is such that dim$\Hom_\Cc (\tau^{-1}T_{\overline{c}},X)=2$ and, in this two-dimensional space, the subspace of morphisms that factor through $\tau T$ has dimension at most one. Then, there is a non-zero morphism in $\Hom_\Cc (\tau^{-1}T_{\overline{c}},X)$ that does not factor through $\tau T$, so $X$ is not associated to a module in ${\CMP}(\B)$. Therefore, $M_\delta\notin\ind\,\underline{\CMP}(\B)$ for all $X_\delta$ in $A_3$. The only remaining positions are those labeled by $\clubsuit$ and $\ast$.
\begin{figure}
\centering
\def\svgwidth{6.2in}
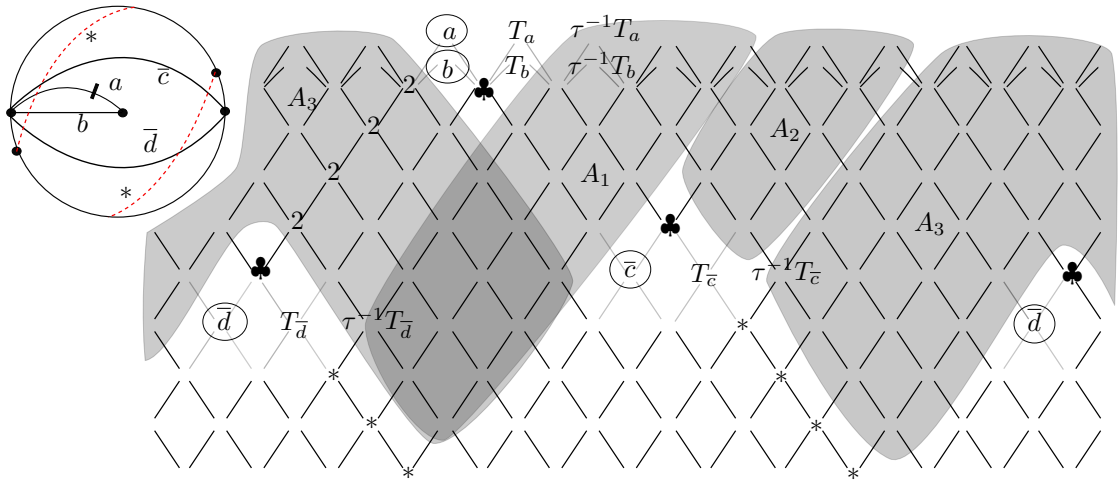
\caption{Lemma \ref{lema2y3} for a triangulation of type III. The position of $T_{a}$, $T_{b}$, $T_{\overline{c}}$ and $T_{\overline{d}}$ make $\ast$ and $\clubsuit$ the only admissible positions for arcs such that  $\gamma $ crosses $\overline{c}$ or $\overline{d}$ and $M_\gamma\in\ind\,\underline{\CMP}$.} 
\label{figura3}
\end{figure}
\end{proof}
The objects labeled by $\clubsuit$ in Figure \ref{figura3} are associated to non-projective indecomposable syzygies, since they are summands of the radical of $P(a)$, $P(\overline{c})$ and $P(\overline{d})$. We will denote the set of these modules by $\underline{\CMP}_\clubsuit$. It is easy to check that they are the syzygies in a periodic projective resolution. The corresponding arcs are shown in Figure \ref{ultimo1}.
\begin{figure}
\centering
\def\svgwidth{3.5in}
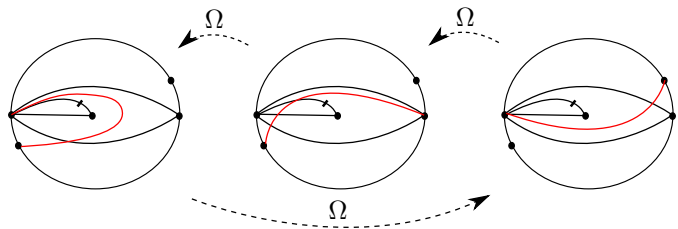
\caption{Arcs associated to the modules in $\underline{\CMP}_\clubsuit \subset\ind\,\underline{\CMP}(\B)$ and the action of $\Omega$.}
\label{ultimo1}
\end{figure}

\subsection {Triangulations of type I and the set $\underline{\CMP}_\odot$}\label{typeI}
In this subsection, we determine all the non-projective syzygies arising from the area that contains the puncture in a triangulation of type I. First, we need to introduce some notation. 
Recall that the AR translation $\tau$ corresponds to moving endpoints clockwise along the boundary and changing the tag at the puncture. If $q$ is a marked point on the boundary, we denote by $\tau(q)$ the next marked point in the clockwise direction.
\begin{df}\label{defcolored} In a triangulation of type I, a marked point $q$ is called \emph{colored} if it is on the boundary and $q$ or $\tau (q)$ is adjacent to the puncture.  
\end{df}
A triangle is \emph{internal} if none of its sides is a boundary segment. Given a triangulation of type I, label from $1$ to $m$, advancing clockwise, the arcs incident to the puncture. We consider indices modulo $m$. If the arcs $k, k+1$, for $k \in \lbrace 1,\ldots ,m\rbrace$, are sides of an internal triangle $\Delta$, we label the third side of $\Delta$ by $\overline{k}$. 
\\
Let $\Ee$ be the set of all internal triangles which have the puncture as a vertex. Two triangles in $\Ee$ are called \emph{consecutive} if they share an edge. If all triangles at the puncture are internal, the set $\Ee$ has exactly $m$ triangles. In this case all colored points are labeled red and blue, and we set $d=m$. If not all triangles at the puncture are internal, then $\Ee$ can be decomposed as $\Ee = \Ee_1 \cup \ldots \cup \Ee_l$ and each $\Ee_i =\lbrace  \Delta_1^{i},\ldots, \Delta^{i}_{t_i} \rbrace$ is such that $\Delta^{i}_{j+1}$ and $\Delta^{i}_{j}$ are consecutive, $\Delta^{i}_{j+1}$ follows $\Delta^{i}_{j}$ in clockwise order and $\Ee_i$ is maximal in the sense that if there is a triangle $\Delta$ in $\Ee$ consecutive to a triangle in $\Ee_i$, then $\Delta$ is in $\Ee_i$. For this case set $d_i = \vert \Ee_i \vert -1$ and $d =\sum d_i$. For example, in the left picture in Figure \ref{ejem}, we have $m=6$ because there are $6$ arcs incident to the puncture, and $\Ee=\Ee_1 \cup  \Ee_2$, where $\Ee_1 =  \lbrace [1, \overline{6}, 6] \rbrace$ and $\Ee_2 =  \lbrace [3, \overline{2}, 2] , [4,\overline{3},3], [5,\overline{4},4]\rbrace$, thus $d=2$.

\begin{df}\label{defredblue} A colored point $q$ is \emph{red} if $\tau (q)$ is not the last vertex in the last triangle in a set $\Ee_i$. A colored point $q$ is \emph{blue} if $q$ is not the first vertex of a first internal triangle in a set $\Ee_i$. 
\end{df} 
 In other words, there are the following three possibilities for a colored point $q$.
 \begin{itemize}
\item[(rb)] $q$ is {\em red and blue} if $q $ is adjacent to the puncture but $q$ is not the first vertex of the first triangle in a set $\Ee_i$, or $\tau(q)$ is adjacent to the puncture but $\tau(q)$ is not the last vertex of the last triangle in a set $\Ee_i$;
\item[(r)] $q$ is {\em red but not blue} if  $q$ is the first vertex of the first triangle in a set $\Ee_i$; in this case $q$ and $\tau^{-1}(q)$ are incident to the puncture.
\item[(b)] $q$ is {\em blue but not red} if $\tau(q)$ is the last vertex of the last triangle in a set $\Ee_i$; in this case $\tau(q)$ and $\tau^{2}(q)$ are incident to the puncture.
\end{itemize}

There are $m+d$ red points, and $m+d$ blue points. In the following, we consider indices modulo $m+d$. 
 We shall define a labeling of the colored points, assigning the labels $r_1,r_2,\ldots,r_{m+d}$ to the red points and the labels $b_1,b_2,\ldots,b_{m+d}$ to the blue points. These labels will be determined by the choice of the labels $r_1$ and $b_1$; the other labels will follow the order of the red or blue points along the boundary in clockwise direction. The labels $r_1$ and $b_1$ are assigned as follows.

\begin{enumerate}
\item If there is no internal triangles (thus $\Ee=\emptyset$), take any red-blue point and label it $r_1$ and $b_1$;
\item If all triangles at the puncture are internal (thus $\Ee=\Ee_1$), take any red-blue point and label it $r_1$ and $b_{m+d}$;
\item 
 If there are internal and non-internal triangles incident to the puncture (thus $\Ee=\Ee_1\cup\cdots\cup\Ee_l$ with $l\ge1$), let $q$ be the first vertex in the first triangle of $\Ee_1$, and label the point $\tau^{-1}(q)$ by $r_1 b_1$, and the point $q$ by $r_2$. Note that $q$ is red but not blue and $\tau^{-1}(q)$ is red and blue.

%If (1) and (2) do not hold, there are internal triangles at the puncture but not all triangles at the puncture are internal.  Then there is an only red point $q$ and it is the first vertex in a first internal triangle $\Delta$ in a set $\Ee_i$. Label the point $\tau^{-1}(q)$ by $r_1 b_1$, and the point $q$ by $r_2$. The next colored points are labeled $r_3 b_2 , \ldots , r_t b_{t-1}$ until an only-blue point $z$ is reached. If $s$ is the last vertex in the last internal triangle in the set $\Ee_i$, then $\tau^{-1}(s)=z$ is the first only-blue and is labeled $b_t$. The point $s$ is labeled $r_{t+1} b_{t+1}$. The next colored points are endpoints of arcs at the puncture, label such point in clockwise order $r_{t+2} b_{t+2}, \ldots r_{t+h} b_{t+h}$ until the first vertex $z'$ in the first triangle in $\Ee_i$ is reached. The point $z'$ is only-red, label this point by $r_{t+h+1}$. Repeat the process until all colored points are labeled. 

\end{enumerate}
See the examples in Figure \ref{ejem}. 
 Note that in the case (1)  every marked point on the boundary is red and blue and each is labeled $r_i,b_i$, for some $i$. In case (2), again all colored points are red and blue, but then each one is labeled $r_i,b_{i-1}$, for some $i$.

In case (3), the sequence of monochromatic points in clockwise order along the boundary is an alternating sequence of red points and blue points. Moreover the labels of the points that are both red and blue  are of the form $r_ib_i$, for some $i$, if the previous monochromatic point is blue, and it is of the form $r_ib_{i-1}$, for some $i$, if the previous monochromatic point is red.
Moreover, in each of the sets $\Ee_j$, the first vertex is red but not blue, all other vertices except for the last are red and blue and have labels $r_ib_{i-1}$, and the last vertex is red and blue and has label $r_ib_i$.

\begin{figure}
\centering
\def\svgwidth{6.2in}
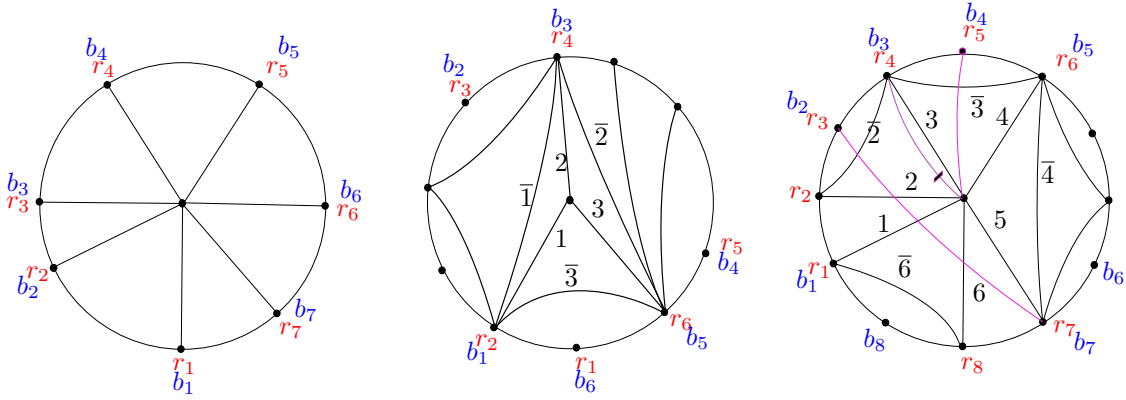
\caption{Triangulations of type I with colored points, from right to left cases (1)(2)(3). On the left example, arcs associated to the modules $M(r_7, b_2)$, $M(r_5, b_4)$ and $M(r_4, b_3)$.}
\label{ejem}
\end{figure}

Let $q$ and $s$ be two marked points on the boundary such that $s \neq q$  and $s \neq \tau (q)$, denote by $\gamma(q,s)$ the arc starting at $q$, going around the puncture in $\tau$ ( = clockwise) direction and ending at $s$. If $s=\tau (q) $ then $\gamma(q,s)$ is a boundary segment. Denote by $\gamma^{\Join}(q,q)$ the notched arc and by $\gamma(q,q)$ the plain arc incident to $q$ and the puncture. 

Now consider the special case where the marked points are colored points.
Suppose first that $r_i$ and $b_j$ are labels in two different marked points. Denote by $M(r_i, b_j)$ the module associated to the arc $\gamma(r_i,b_j)$, if the arc $\gamma(r_i,b_j)\notin \Tt$, and let $M(r_i, b_j)=0$ otherwise. Next suppose that $r_i$ and $b_j$ are labels in the same marked point $q$. If $\gamma(q,q) \in \Tt$, then let $M(r_i,b_j)$ be the module associated to $\gamma^{\Join}(r_i,b_j)$. In the other case, let $M(r_i,b_j)$ be the module associated to $\gamma(r_i,b_j)$. For example, see in Figure \ref{ejem}, $M(r_4,b_3)=M_{\gamma^\Join (r_4,b_3)}$ and $M(r_5,b_4)=M_{\gamma (r_5,b_4)}$.
%\begin{rema}\label{masrema}
%%\begin{enumerate}
%%\item 
%If a marked point $q$ is red and blue then $q$ is labeled $r_i b_i$ or $r_i b_{i-1}$. In particular, if there are no internal triangles at the puncture then all colored points are labeled $r_i  b_i$, and if all triangles at the puncture are internal then all colored points are labeled $r_i  b_{i-1}$.
%%\item If a marked point $q$ is only-blue then $\tau(q)$ is the last vertex in the last internal triangle $\Delta \in \Ee_t$. It follows that $\gamma(\tau(q),\tau(q)) \in\Tt$ and $\gamma(\tau^2(q),\tau^2(q)) \in \Tt$.
%%\item If a marked point $q$ is only-red then $q$ is the first vertex in the first internal triangle $\Delta \in \Ee_t$. It follows that $\gamma(q,q) \in \Tt$ and $\gamma(\tau^{-1}(q),\tau^{-1}(q)) \in \Tt$.  
%%\end{enumerate} 
%\end{rema} 

The modules denoted by $M(r_i,b_j)$ will be important throughout this section. First let us study the modules $M(r_i,b_i)$ and $M(r_i,b_{i+1})$, $1 \leq i \leq m+d$.

\begin{lema}\label{remaii} Let $i\in\{1,2,\ldots, m+d\}.$ Then
\begin{itemize}
\item [\textup{(a)}] $M(r_i,b_{i+1})$  is projective.
\item [\textup{(b)}] $M(r_i,b_i)$ is projective or there exists an arc $\overline{k}=\gamma(q,s)$ such that $M(r_i,b_i)=M_{\gamma(q,\tau^{-1}(s))}.$
\end{itemize}
\end{lema}

\begin{proof}

(a) 
Let $q$ be the marked point labeled $r_i$.
Suppose first that $q$ is the first vertex in an internal triangle $\Delta \in \Ee_t$.   Let $k$ be the arc in $\Delta$ from $q$  to the puncture and  let $\overline{k}=\gamma(q,s)$ be the arc in $\Delta$ that is not incident to the puncture. Then $\tau^{-1}(s)$ and $s$ are the next colored points, their blue labels are $b_{i}$ and $b_{i+1}$ respectively, and so $M(r_i,b_{i+1})=M_{\overline{k}}$ is zero, hence projective.

Now suppose that $\tau(q)$ is the first vertex in an internal triangle $\Delta \in \Ee_t$. Let $\overline{k} = \gamma(\tau (q),s)$ be the arc in $\Delta$ that is not incident to the puncture. The colored points after $q$ are $\tau(q)$ and $\tau^{-1}(s)$. If $q$ is not adjacent to the puncture, then the blue labels of $\tau(q)$ and $\tau^{-1}(s)$ are $b_i$ and $b_{i+1}$, respectively. If $q$ is adjacent to the puncture then $q$ is labeled $r_i b_i$ and $\tau^{-1}(s)$ is labeled $b_{i+1}$. In both cases, $\gamma(r_i,b_{i+1})=\tau^{-1}(\overline{k})$ and $M(r_i,b_{i+1})=P(\overline{k})$ is projective.

It remains the case where there both $q$ and $\tau(q)$ are incident to the puncture and $\tau(q) $ is not a vertex of an internal triangle. Then $q$ is labeled $r_i b_i$ and the blue label of $\tau(q)$ is  $b_{i+1}$. Hence the curve $\gamma(r_i,b_{i+1})$ is a boundary segment and thus $M(r_i,b_{i+1})$ is zero, hence projective.

(b) Suppose $M(r_i,b_i)$  is not projective and 
let $q$ be the marked point labeled $r_i$. 
Suppose first that $\tau(q)$ is the last vertex in an internal triangle $\Delta$ in $\Ee_t$. Then $q$ is labeled $r_i, b_{i-1}$ and the blue label of $\tau(q)$ is $b_i$. Then, the curve $\gamma(r_i,b_i)$ is a boundary segment and $M(r_i,b_i)$ is zero, hence projective.

Now suppose that both $q$ and $\tau(q)$ are incident to the puncture. It follows from the construction of the labels that $q$ is labeled $r_ib_i$. Thus $M(r_i,b_i)=M_{\gamma^{\bowtie}(q,q)}=P(k)$, where $k$ is the arc from $\tau(q)$ to the puncture. This is a contradiction to our assumption that  $M(r_i,b_i)$ is not projective.

It remains the case where $q$ is the first vertex of an internal triangle $\Delta$. Let $\overline{k}=\gamma(q,s)$ be the side of $\Delta$ that is not incident to the puncture. Then the point $\tau^{-1}(s)$ is labeled $b_i$ and    $M(r_i,b_i)=M_{\gamma(q,\tau^{-1}(s))}$.
\end{proof}

We are now ready for the first main result of this section. The following theorem determines most of the modules in $\underline{\CMP}(\B)$ over configurations of type I and describes the action of $\Omega$ in geometric terms. We define the set $\underline{\CMP}_\odot$ by 

\[ \underline{\CMP}_\odot = \{M(r_i, b_j) \mid i=1,2,\ldots, m+d,  j= i+2,\ldots, i-1\}.\]

\begin{teo}\label{unlema} With the above notation, if $j \in \lbrace i+2,..., i-1\rbrace$, then
\begin{equation} \label{omega}
\Omega M(r_i, b_j) \simeq M(r_{j-1}, b_i).
\end{equation}
In particular, the map $\Omega \colon \underline{\CMP}_\odot \raw \underline{\CMP}_\odot$ is a bijection and $\underline{\CMP}_\odot \subset \ind\, \underline{\CMP}(\B)$.
\end{teo}
\begin{proof}
We prove the statement by analyzing the possible relative positions of the points $r_i$, $b_i$ and $r_{j-1}$, $b_j$ where $i$ is arbitrary and $j= i+1, \ldots, i-1$. The possibilities are listed in Figure \ref{cuadro}. The first row of that figure shows possibilities for $r_i$ and $b_i$ and the second row shows possibilities for $r_{j-1}$ and $b_j$. The last row illustrates three special cases. 

Choosing a module $M(r_i,b_j)$ corresponds to either choosing one of the three cases 1,2,3 of the first row and one of the three cases a,b,c of the second row, or choosing one of the three cases d,e,f of the third row. In each case the syzygy $\Omega M(r_i,b_j)$ is easily computed and the results are shown in  Figures \ref{figu1}-\ref{figu3} at the end of the paper. We present the computation in detail for the cases (1c), (2b) and (d), and leave the other to the reader. Notice that the formula also works for the configuration of type I with only two arcs incident to the puncture; that case is also left to the reader.

Throughout the proof we use the notation $\beta_k$ for the maximal non-zero path starting at the vertex $k'\neq k$ such that there is an arrow $\overline{k} \raw k'$. Note that $k'$ is unique if it exists. 

Case (1c): the module $M(r_i,b_j)$  is given by the string $k \raw  \cdots \raw k+t \law \overline{k+t}  \raw \beta_{k+t}$ and top$ M(r_i,b_j) = S(k)\oplus S(\overline{k+t})$, where $k$ is the first arc crossed by $\gamma(r_i,b_j)$ and $k+t$ is the last arc incident to the puncture that is crossed by $\gamma(r_i,b_j)$. The projective cover $P(k) \oplus P(\overline{k+t})$ of $M(r_i,b_j)$ consist of a uniserial projective $P(k)$ given by the string $k \raw k+1 \raw \cdots \raw k-2$ and a projective $P(\overline{k+t})$ given by the string $\beta_{k+t-1}^{-1} \law \overline{k+t-1} \law k+t \law \overline{k+t}  \raw \beta_{k+t}$. Therefore, the syzygy $\Omega M(r_{i},b_j)$ is given by the string $\beta_{k+t-1}^{-1} \law k+t \raw k+t+1 \raw \cdots \raw k-2$, this module is denoted by $M(r_{j-1},b_i)$, see Figure \ref{figu1}.

Case (2b): the module $M(r_i,b_j)$ is given by the string $k \raw \cdots \raw k+t \law \overline{k+t} \raw \beta_{k+t}$ and top$ M(r_i,b_j) = S(k)\oplus S(\overline{k+t})$, where $k$ is the first arc crossed by $\gamma(r_i,b_j)$ and $k+t$ is the last arc incident to the puncture that is crossed by $\gamma(r_i,b_j)$. The projective cover $P(k)\oplus P(\overline{k+t})$ of $M(r_i,b_j)$ consists of a projective $P(\overline{k+t})$ given by the string $\beta_{k+t}^{-1} \law \overline{k+t} \raw k+t$ and a  projective $P(k)$ supported on the paths $k \raw \overline{k-1} \raw \beta_{k-1}$ and $k \raw k+1 \raw \cdots \raw k-1 = k \raw \overline{k-1} \raw k-1$ so the dimension vector of $P(k)$ is $(d_i)_{i \in Q_0}$ with $d_i=1$ for $i=k, k+1,\ldots,k-1,\overline{k-1}$ and for all $i$ in the support of $\beta_{k-1}$, and $d_i=0$ for all other $i$. Then $\Omega M(r_i,b_j)$ is given by the string $\beta_{k-1}^{-1} \law \overline{k-1} \raw k-1 \law k-2 \law \cdots \law k+t$. Thus, $\Omega M(r_i,b_j)=M(r_{j-1},b_i)$. See Figure \ref{figu2}.

The last row in Figure \ref{cuadro} contains three special cases, where the projective is not like in the other cases.

Case (d): the module $M(r_i,b_j)$ is given by the string $\overline{k} \raw \beta_{k} $ and top$ M(r_i,b_j)=S(\overline{k})$. The projective cover $P(\overline{k})$ is a biserial projective given by the string $\beta_{k-1}^{-1} \law \overline{k-1} \law k \law \overline{k} \raw \beta_{k}$. Then, $\Omega M(r_i,b_j)$ is given by the string $k \raw \overline{k-1} \raw \beta_{k-1}$, so $\Omega M(r_i,b_j)= M(r_{j-1},b_i)$.
This shows equation (\ref{omega}).

In order to show the last statement of the theorem, note that
$\Omega$ is bijective in the set $\underline{\CMP}_\odot$ because, for the indices under consideration, the map $(r_i,b_j) \raw (r_{j-1},b_i)$ is bijective. Each pair $(r_i,b_j)$ has a unique predecessor $(r_j,b_{i+1})$ and the indices satisfy the condition $i+1 \in \{ j+2, \ldots j-1\}$ since $j \neq i+1$ and $j \neq i$. 
Moreover, we see that for each pair $(r_i,b_j)$ the module $M(r_i,b_j)$ is not zero and the projective cover $P(M(r_i,b_j))$ is different from $M(r_i,b_j)$. Thus $M(r_i,b_j)$  is not projective and its syzygy is not trivial. This shows that  $\underline{\CMP}_\odot \subset \ind\,\underline{\CMP}(\B)$.
\end{proof}

\begin{figure}\label{cosas}
\centering
\def\svgwidth{6.4in}
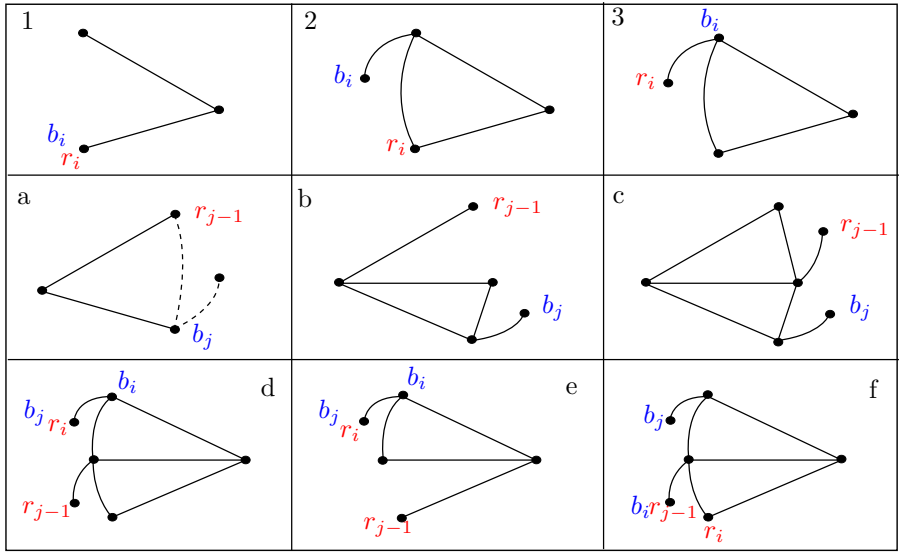
\caption{Possible relative positions of $r_i$, $r_{j-1}$, $b_i$ and $b_j$.}
\label{cuadro}
\end{figure}
%\begin{figure}
%\centering
%\def\svgwidth{5in}
%\input{casos1.pdf_tex}
%\caption{Cases (1a), (1b) and (1c). The left column shows part of the triangulation, the middle column, the module $M(r_i,b_j)$ and the right column shows $\Omega M(r_i,b_j)=M(r_{j-1},b_i)$.}
%\label{figu1}
%\end{figure}
%\begin{figure}
%\centering
%\def\svgwidth{5in}
%\input{casos2.pdf_tex}
%\caption{Cases (2a), (2b) and (2c). The left column shows part of the triangulation, the middle column, the module $M(r_i,b_j)$ and the right column shows $\Omega M(r_i,b_j)=M(r_{j-1},b_i)$.}
%\label{figu2}
%\end{figure}
%\begin{figure}\label{fig3}
%\centering
%\def\svgwidth{5in}
%\input{casos3.pdf_tex}
%\caption{Cases (3a), (3b) and (3c). The left column shows part of the triangulation, the middle column, the module $M(r_i,b_j)$ and the right column shows $\Omega M(r_i,b_j)=M(r_{j-1},b_i)$.}
%\label{figu3}
%\end{figure}
%\begin{figure}\label{fig3}
%\centering
%\def\svgwidth{5in}
%\input{casos4.pdf_tex}
%\caption{Cases (d), (e) and (f).}
%\label{figu3}
%\end{figure}

\subsection{Full description of the category $\underline{\rm CMP}(\B)$}
In Theorem \ref{unlema}, we gave an explicit description of the subset  $\underline{\CMP}_\odot$ of $\ind\, \underline{\CMP}(\B)$.
Although this is the most interesting part of $ \underline{\CMP}(\B)$, it is in general not all of it. In this subsection, we describe  $ \underline{\CMP}(\B)$ completely.

Before formulating our main result, we need two preparatory lemmas.

Let $\Tt$ be a triangulation of type I and
let $\Delta$ be an internal triangle in $\Tt$ incident to the puncture. As usual we denote the sides of $\Delta$ by $k,k+1$ and $\overline{k}$, where $\overline{k}$ is the side opposite to the puncture.

\begin{lema}\label{lema02}  Let $\overline{k}=\gamma(q,s)$ be an arc as above and 
let $\gamma$ be an arc that crosses the arcs $\overline{k}=\gamma(r,s)$ and $\tau^{-1} (\overline{k})$. 
If $M_\gamma\in  \CMP(\B)$ then  $\gamma =\gamma(\tau^{-1}(s),q)$. 

\end{lema}

\begin{proof}
We work in the AR quiver of the cluster category $\mathcal{C}$, see Figure \ref{label2}. 
The relative positions of $T_{\overline{k}}, T_k, T_{k+1}$  are indicated in the figure. The condition that the arc $\gamma $ crosses $\overline{k}$ and $\tau^{-1}(\overline{k})$ translates to the condition that $\Hom_\Cc(T_{\overline{k}},X_\gamma)\ne 0$ and  $\Hom_\Cc(\tau^{-1}T_{\overline{k}},X_\gamma)\ne 0$. It then follows from the structure of the AR quiver
that $X_\gamma$ must lie in the shaded region in Figure \ref{label2}.
Suppose first that $X_\gamma \in A_1$ and let $f$ be a non-zero morphism in $\Hom_\Cc(\tau^{-1} T_{\overline{k}}, X_\gamma)$, then it cannot factor through $\tau T$ because, by the position of $T_{\overline{k}}$, $A_1$ does not contain summands of $\tau T$. Thus, there is a non-zero morphism in $\Hom_\B(\tau^{-1}P(\overline{k}),M_\gamma)$ and $M_\gamma\notin {\CMP(\B)}$ by Remark~\ref{rema util} (a). Now suppose that $X_\gamma \in A_2$ and let $f $ be a non-zero morphism in $\Hom_\Cc(\tau^{-1} T_{k+1}, X_\gamma)$, then it cannot factor through $\tau T$ because, by the position of $T_{\overline{k}}$ and $T_{k+1}$, $A_1$ and $A_2$ do not contain summands of $\tau T$. Thus, there is a non-zero morphism in $\Hom_\B(\tau^{-1}P(k+1),M_\gamma)$ and $M_\gamma\notin  \CMP(\B)$ by Remark~\ref{rema util} (a). The only other object associated to an arc crossing the arcs $\overline{k}$ and $\tau^{-1}(\overline{k})$ is $X_{\gamma(\tau^{-1}(s),q)}$. The lemma follows.
\end{proof}

\begin{figure}
\centering
\def\svgwidth{6.3in}
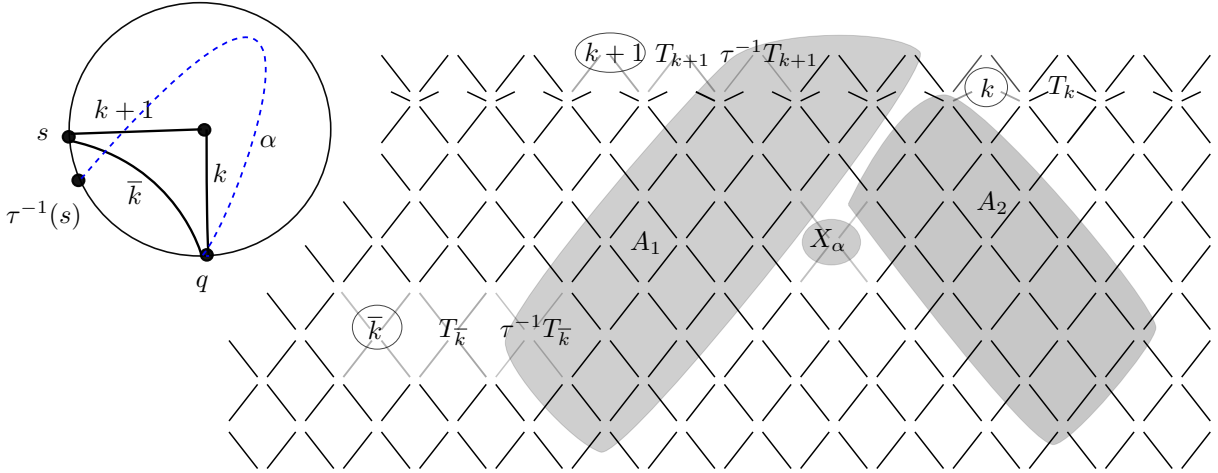
\caption{Proof of Lemma \ref{lema02}. The arc $\gamma(\tau^{-1}(s),q)$ is denoted by $\alpha$ in the figure.}
\label{label2}
\end{figure}

\subsubsection{The sets $\Xx_k$ and $\Mm_k$}
Now, we set the last results needed to give a complete geometric description of $\underline{\CMP}(\B)$. Recall the Definition \ref{types}. Let $\Tt$ be a triangulation and let $\overline{k}$ be an arc such that
\begin{enumerate}
\item[(a)] $\overline{k}$ is a side of an internal triangle such that the other two sides are incident to the puncture, if $\Tt$ is of type I,
\item[(b)] $\overline{k}$ is the arc $\overline{a}$ in Figure \ref{2 y 3} when $\Tt$ is of type II, 
\item[(c)] $\overline{k}$ is one of the arcs $\overline{c}$ or $\overline{d}$ in Figure \ref{2 y 3} when $\Tt$ is of type III.
\end{enumerate}
By cutting along $\overline{k}=\gamma(q,s)$, we obtain two pieces of the disc. Denote by $R_k$ the piece that does not contain the puncture. See the Figure \ref{labelG}.

%\begin{figure}
%\centering
%\def\svgwidth{3.6in}
%\input{ejemplok.pdf_tex}
%\caption{Arc $\overline{k}$ and region $R_k$ shaded in a triangulation of type I (left), and of type II (right).}
%\label{ejemplok}
%\end{figure}

Let $\Xx_k$ be the set of objects $X_\gamma$ in the cluster category such that $\gamma$ lies in $R_k$, or $\tau (\gamma)$ lies in $R_k$. The objects in $\Xx_k$ lie in the shaded region in Figure \ref{labelG}. The objects $X_\gamma \in \Xx_k$ such that $\gamma \notin R_k$ and $\tau (\gamma)\in R_k$ are the objects marked with a bullet together with $T_{\overline{k}}$.
\\
Let $\Mm_k$ be the set of modules $M_\gamma$ such that $X_\gamma \in \Xx_k$.

\begin{figure}
\centering
\def\svgwidth{6in}
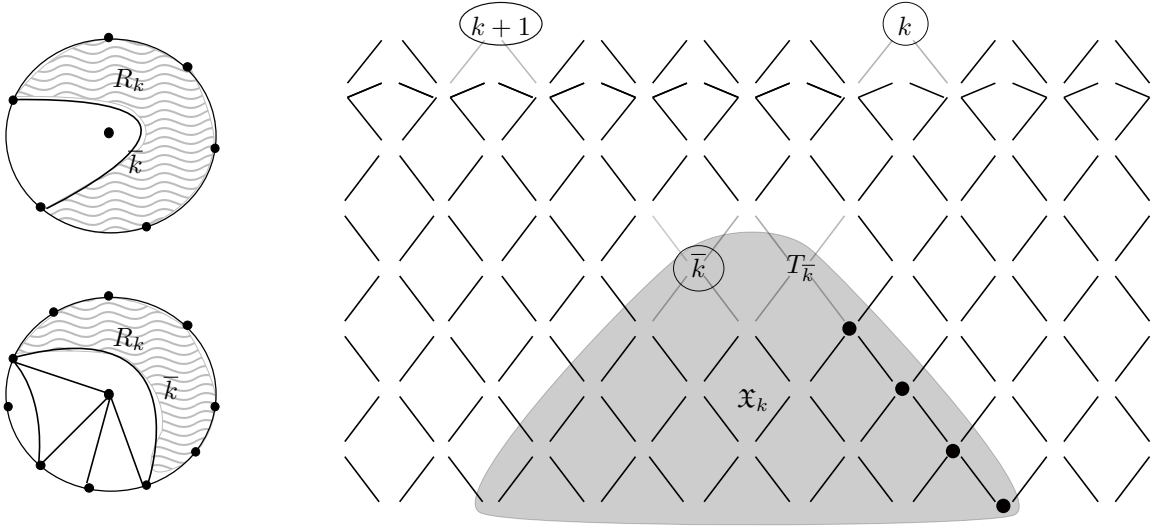
\caption{On the right Region $R_k$ shaded. On the left, the region shaded in the AR quiver shows the objects in $\Xx_k$.}
\label{labelG}
\end{figure}

\subsubsection{ The set $\underline{\CMP}_\Delta$}

Recall that we denote by $\Ee$ the set of all internal triangles that are incident to the puncture. We now consider the other internal triangles.
For every internal triangle $\delta\notin\Ee$, we construct 3 arcs by performing an elementary move in counterclockwise direction to each of the sides of $\delta$ such that the resulting arc runs through the interior of $\delta$. 
We denote by $\underline{\CMP}_\Delta$ the set of all modules $M_\gamma$ for which there exists an internal triangle $\delta\notin\Ee$ such that the arc $\gamma$ is one of the three arcs given by the above construction with respect to $\delta$. As we have seen in section~\ref{sect 4}, we have $\underline{\CMP}_\Delta \subset\ind\,\underline{\CMP}(\B)$. In fact 
$M_\gamma$ is a radical summand of the projective $P(a)$ where $a$ is the side of $\delta $ that is crossed by $\gamma$.

\begin{lema}\label{lemadelta}  If a module $M_\gamma \in \Mm_k$ is a non-projective syzygy then $M_\gamma\in\underline{\CMP}_\Delta$
\end{lema}

\begin{proof}
Case(i): Let $M_\gamma$ be an indecomposable non-projective syzygy such that $\gamma$ lies entirely in $R_k$ for some $\overline{k}=\gamma(q,s)$. Since $\overline{k} \in \Tt$, the only arcs $x\in \Tt$ such that the projective module $P(x)=M_{\tau^{-1}(x)}$ is supported in the vertices corresponding to arcs in $R_k$ are those that lie entirely in $R_k$ and those starting at $s$. Since $M_\gamma$ is a syzygy, there is a monomorphism $u \colon M_\gamma \raw P= \oplus_{i=1}^n P(x_i)$ such that each arc $x_i$ lies on $R_k$ or has $s$ as an endpoint. Suppose that an indecomposable summand $P(x)$ of $P$ is associated to an arc that does not lie in $R_k$. Let $N$ be the maximal submodule of $P(x)$ such that its corresponding arc $\delta_N$ lies in $R_k$, then $N$ is a summand of rad$P(\overline{k})$. The morphism $u$ factors through $(P/P(x))\oplus N$ and $M_\gamma$ can be considered a submodule of $(P/P(x))\oplus P(\overline{k})$ via $M_\gamma \hookrightarrow (P/P(x))\oplus N \hookrightarrow (P/P(x))\oplus P(\overline{k})$.
So, without loss of generality, we can assume that  all the arcs $x_i$ lie in $R_k$. Now, define the triangulated polygon $\Pi_k$ by gluing to $R_k$, along $\overline{k}$, a triangle with two boundary segments as the other sides. Thus, $\Pi_k$ has a triangulation $\widetilde{\Tt}$ that is formed by a triangulated region $\widetilde{R_k}$, analogous to $R_k $, and an extra triangle, see Figure \ref{label4}. For each arc $j$ in $R_k$, denote by $\tilde{j}$ the analogous arc in $\widetilde{\Tt}$. Denote by $\widetilde{\B}$ the cluster-tilted algebra of type $\mathbb{A}$ arising form the triangulation $\widetilde{\Tt}$, \cite{ABCP,CCS}. Since $M_\gamma$ is not projective over $\B$, $\gamma$ is not of the form $\tau^{-1}(a)$ for an arc $a \in \Tt$. This implies that $\tilde{\gamma}$ is not $\tau^{-1} (\tilde{a})$ for an arc in $\widetilde{\Tt}$. Then, $M_{\widetilde{\gamma}}$ is not projective over $\widetilde{\B}$. Define the induced morphism of representations $\widetilde{u} =(u_{\tilde{j}}): M_{\widetilde{\gamma}} \hookrightarrow \oplus_{i=1}^n P(\widetilde{x_i})$ by $u_{\tilde{j}}=u_j$, for all $\tilde{j} \in \widetilde{\Tt}$. By construction, $\widetilde{u}$ is a monomorphism in $\mod \widetilde{\B}$. Then, $M_{\widetilde{\gamma}}$ is an indecomposable non-projective syzygy and, by Section 4.1, it has to be a summand of the radical of an indecomposable projective module $P({\widetilde{a}})$, where $\widetilde{a}$ is a side of an internal triangle $\widetilde{\Delta}$. Lifting the arc $\widetilde{\gamma}$ to $\Tt$, we obtain that $M_{\gamma}$ is a summand of the radical of a projective module $P(a)$, where $a$ is a side of the lifted internal triangle $\Delta$ in $R_k$. Thus $M_\gamma\in\underline{\CMP}_\Delta$.
\begin{figure}
\centering
\def\svgwidth{3.5in}
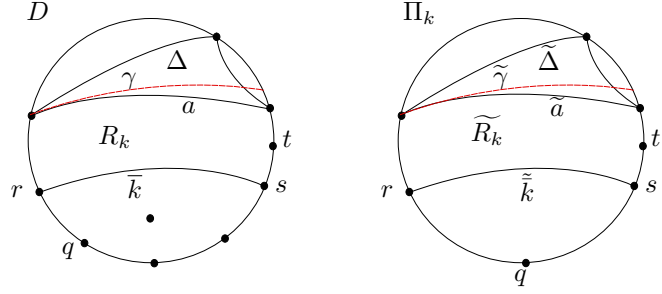
\caption{Original punctured disc $D$ and polygon $\Pi_k$.}
\label{label4}
\end{figure}

Case(ii): Now, consider an arc $\gamma$ that starts at $q$, crosses $\overline{k}$, continues in $\tau$ direction and ends at a vertex in $R_k$, and let $M_\gamma$ be the associated module. By Remark \ref{rema11}, since $M_\gamma \in \ind\,\underline{\CMP}(\B)$ we have $\tau M_\gamma \in\ind\, \underline{\CMI}(\B)$ and $\tau M_\gamma =M_{\tau(\gamma)}$ is a non-injective co-syzygy in $\mod \B$. Observe that $\tau(\gamma)$ lies in $R_k$. In particular $\tau(\gamma) \neq \overline{k}$ because $M_\gamma$ is not projective. Applying the same ideas as in case (i), using injective modules and taking $M_{\tau(\gamma)}$ as a co-syzygy, we find an epimorphism $h=(h_j)_{j\in Q}: \oplus_{i=1}^{l}I(x_i) \raw M_{\tau(\gamma)}$ such that each arc $x_i$ lies in $R_k$. Again, $h_j \neq 0$ only for $j$ in $R_k$, and now $h_j$ is a surjective linear map. Considering the induced morphism $\widetilde{h}$ in $\mod \widetilde{\B}$, we have that $M_{\widetilde{\tau(\gamma)}}$ is an indecomposable co-syzygy over $\mod \widetilde{\B}$. Then, applying Remark \ref{rema11} to the geometric interpretation in Section 4.1, we conclude that $M_{\widetilde{\tau(\gamma)}}$ is a summand of $I(\widetilde{a})/S(\widetilde{a})$, where $\widetilde{a}$ is a side of an internal triangle $\widetilde{\Delta}$ in $\Pi_k$. Lifting the arc $\widetilde{\tau(\gamma)}$ to $D$, we obtain that $M_{\tau(\gamma)}$ is a summand of $I(a)/S(a)$, where $a$, the lift of $\widetilde{a}$, is a side of the internal triangle $\Delta$. Then $\tau^{-1}M_{\tau(\gamma)}= M_\gamma$ is a summand of the radical of $P(a)$.
\end{proof}

\subsubsection{The main result}
We are now ready for the full description of $\underline{\CMP}(\B)$. The following result, together with the morphisms description in the next subsection, gives a complete description of the category $\underline{\CMP}(\B)$ in geometric terms.
\begin{teo}\label{teofull} Let $\B$ be a cluster-tilted algebra of type $\mathbb{D}$ and let $\mathcal{T}$ be the corresponding triangulation of the punctured disk. 
Let $t$ be the number of internal triangles which do not have the puncture as a vertex. Then, the indecomposable modules in $\underline{\CMP}(\B)$ are precisely
\begin{enumerate}
\item $\underline{\CMP}_\odot \cup \underline{\CMP}_\Delta$ if $\Tt$ is of type I,
\item $ \underline{\CMP}_\Delta$ if $\Tt$ is of type II,
\item $\underline{\CMP}_\clubsuit \cup \underline{\CMP}_\Delta$ if $\Tt$ is of type III.
\end{enumerate}
In particular, the number of indecomposable modules in $\underline{\CMP}(\B)$ is

\[ \begin{array}{cl} (m+d)(m+d-2)+3t &\textup{if $\Tt$ is of type I,}\\
3t &\textup{if $\Tt$ is of type II},\\
3(t+1)&\textup{if $\Tt$ is of type III.}
\end{array}\]

\end{teo}

\begin{proof}
 (1) Let $\mathcal{T}$ be of type I, and let $\gamma=\gamma(q,s)\notin \mathcal{T}$ be an arc such that $M_\gamma\in\underline{\CMP}(\B)$ is not zero.
 Suppose that $M_\gamma\notin\underline{\CMP}_\odot$. Then, according to the definition of $\underline{\CMP}_\odot$, there are 4 possible cases.
 \begin{itemize}
\item [(i)] The point $q$ is uncolored or  $q$ is blue but not red.
\item [(ii)] The point $s$ is uncolored or $s$ is red but not blue.
\item [(iii)] The point $q$ is labeled $r_i$ and the point $s$ is labeled $b_{i+1}$. 
\item [(iv)]  The point $q$ is labeled $r_i$ and the point $s$ is labeled $b_{i}$. 
\end{itemize}
  (i) Suppose first that $q$ is  blue but not red. By condition (r) of Definition \ref{defredblue}, the triangulation $\mathcal{T}$ contains an arc $a$ from $\tau(q)$ to the puncture and an arc $b$ from $\tau^2(q)$ to the puncture. If $q$ is equal to $s$, then either $M_\gamma =P(a)$ if $\gamma $ is notched, or $M_\gamma = \tau^{-1} P(b)$ if $\gamma $ is plain. In the first case, $M_\gamma$  is zero in $\underline{\CMP}(B)$, because it is projective, and in the second case, $M_\gamma\notin \CMP(B)$, by Remark \ref{rema util}(a). This is a contradiction, thus $q$ must be different from $s$.  Again by condition (r) of Definition \ref{defredblue}, the point $\tau (q)$ is a last vertex in a last triangle $[k,\overline{k},k+1]$ in a set $\Ee_i$ and the arc $\gamma(\tau^2(q),\tau^2(q))=k+2$ belongs to $\Tt$. The module $M_\gamma$ is defined by a string $\beta_{k } \law \overline{k} \law k+1 \raw w$, where $w$ is a walk. The string module $\tau^{-1} P(k+2)= M(\beta_{k } \law \overline{k}) $ is a submodule of $M_\gamma$. Therefore there is a non zero morphism in $\Hom_\B (\tau^{-1} P(k+2),M_\gamma)$, and by Remark \ref{rema util}(a), $M_\gamma$ is not in $  \CMP (\B)$, a contradiction.
  
  Now suppose that $q$ is uncolored. By Definition \ref{defcolored}, this means that neither $q$ nor $\tau(q)$ is incident to the puncture in $\mathcal{T}$. Since $\mathcal{T} $ is a triangulation, this implies that the arc from $q$ to the puncture and the arc from $\tau(q)$ to the puncture crosses an arc in $\mathcal{T}$. This implies that there exists an arc $\overline{k}=\gamma(x,y)\in\mathcal{T}$ that is not incident to the puncture such that $q, \tau(q)\in R_k\setminus\{x,y\}$.
  If $\gamma$ or $\tau(\gamma)$ lies entirely inside $R_k$ then $M_\gamma\in\Mm_k$ and Lemma \ref{lemadelta} implies that $M_\gamma\in\underline{\CMP}_\Delta$.  
On the other hand, if both $\gamma$ and $\tau(\gamma)$ do not lie in $R_k$ then both $\gamma $ and $\tau(\gamma)$ cross $\overline{k}=\gamma(x,y)$ and $q,\tau(q)\notin\{x,y\}$. In other words, $\gamma $ crosses both $\overline{k}$ and $\tau^{-1}(\overline{k})$ and $q\notin\{x,y,\tau^{-1}(x),\tau^{-1}(y)\}$. Then Lemma~\ref{lema02} implies that $M_\gamma$ is not in $  \CMP (\B)$, a contradiction.

(ii) Suppose first that $s$ is red but not blue. By condition (b) of Definition \ref{defredblue}, the triangulation $\mathcal{T}$ contains an arc $a$ from $s$ to the puncture and an arc $b$ from $\tau^{-1}(s)$ to the puncture. If $q$ is equal to $s$, then $\gamma $ is an arc from $q$ to the puncture, and since $\gamma$ is not in the triangulation, we see that $\gamma$ must be notched. 
 Then $M_\gamma=I(b)$. But since the projective dimension of $I(b) $ is 0 or 1, we see that $M_\gamma$ is projective or, by Remark \ref{rema util}(b),  $M_\gamma\notin \CMP$. In both cases $M_\gamma\notin \ind\,\underline{\CMP} (\B)$.
 This is   a contradiction, thus $q$ must be different from $s$.  Then, in the AR-quiver of the cluster category, the object $X_\gamma$ lies in one of the positions labeled by a $\ast$  in Figure \ref{label1}. From the relative position of $T_{\overline{k}}$ and the $\ast$, it follows that there is a nonzero morphism $f\in \Hom_\Cc(\tau^{-1} T_{\overline{k}}, X_\gamma)$. Moreover, because of the relative position of  $T_{\overline{k}}$, the shaded region $A_2$ do not contain any summands of $\tau T$, and  because of the relative position of  $T_{k}$, the shaded region $A_1$ do not contain any summands of $\tau T$. This implies that the morphism $f$ does not factor through $\tau T$, and therefore  $f$ induces a nonzero morphism in $\Hom_{\B}(\tau^{-1}P(\overline{k}), M_\gamma)$. Now Remark \ref{rema util}(a) implies that $M_\gamma$ is not in $ \CMP(\B)$, a contradiction.

Now suppose that $s$ is uncolored.  Replacing $q$ by $s$ in the argument  in the case (i) for $q$ uncolored, we see that if  $M_\gamma\notin\underline{\CMP}_\Delta$ then  there exists an arc $\overline{k}=\gamma(x,y)\in\mathcal{T}$ such that  $\gamma $ crosses both $\overline{k}$ and $\tau^{-1}(\overline{k})$ and $s\notin\{x,y,\tau^{-1}(x),\tau^{-1}(y)\}$. Then Lemma~\ref{lema02} implies that $\gamma=\gamma(q,s)=\gamma(\tau^{-1}(y),x)$ and thus $s=x$, a contradiction. 
%
%Now suppose that $s$ is uncolored.  By Definition \ref{defcolored}, this means that neither $s$ nor $\tau(s)$ is incident to the puncture in $\mathcal{T}$. Then there exists an arc $\overline{k}=\gamma(x,y)\in\mathcal{T}$ that is not incident to the puncture such that $s, \tau(s)\in R_k\setminus\{x,y\}$.
% Again, if $\gamma$ or $\tau(\gamma)$ lies entirely inside $R_k$ then $M_\gamma\in\Mm_k$ and Lemma \ref{lemadelta} implies that $M_\gamma\in\underline{\CMP}_\Delta$.  
%On the other hand, if both $\gamma$ and $\tau(\gamma)$ do not lie in $R_k$ then both $\gamma $ and $\tau(\gamma)$ cross $\overline{k}=\gamma(x,y)$ and $s,\tau(s)\notin\{x,y\}$. In other words, $\gamma $ crosses both $\overline{k}$ and $\tau^{-1}(\overline{k})$ and $s\notin\{x,y,\tau^{-1}(x),\tau^{-1}(y)\}$. Then Lemma~\ref{lema02} implies that $\gamma=\gamma(q,s)=\gamma(\tau^{-1}(y),x)$ and thus $s=x$, a contradiction.

(iii) Let $\gamma=\gamma(q,s)$ and suppose $q$ is labeled $r_i$ and $s$ is labeled $b_{i+1}$. 
Whenever $M_\gamma=M(r_i,b_{i+1})$, we conclude from Lemma~\ref{remaii}(a) that $M_\gamma$ is projective and thus $M_\gamma$  is zero in $\underline{\CMP}(\B)$, a contradiction.

If $\gamma $ is not incident to the puncture then $M_\gamma=M(r_i,b_{i+1})$, a contradiction. 

Suppose that $\gamma $ is incident to the puncture. If $\gamma $ is plain then  $M_\gamma=0$ if $\gamma\in\mathcal{T}$ and $M_\gamma=M(r_i,b_{i+1})$ otherwise. Both cases yield a contradiction. On the other hand, if $\gamma $ is notched, let $\gamma^\circ$ denote the plain arc homotopic to $\gamma$. Then if $\gamma^\circ\in\mathcal{T}$, we have $M_\gamma=M(r_i,b_{i+1})$, a contradiction. And if $\gamma^\circ\notin\mathcal{T}$, then $\tau(\gamma)$ is the plain arc from $\tau(q)$ to the puncture, and since $q$ is a colored point, we get $\tau(\gamma)\in\mathcal{T}$. This implies that $M_\gamma$ is projective and thus $M_\gamma$  is zero in $\underline{\CMP}(\B)$, a contradiction.

(iv) Now let $\gamma=\gamma(q,s)$ and suppose $q$ is labeled $r_i$ and $s$ is labeled $b_{i}$. 
Whenever $M_\gamma=M(r_i,b_{i})$, we conclude from Lemma~\ref{remaii}(a) that $ M_\gamma \in \underline{\CMP}_\Delta$ and we are done.

If $\gamma $ is not incident to the puncture then $M_\gamma=M(r_i,b_{i})$.

Suppose that $\gamma $ is incident to the puncture. Since $q=s$ is a red and blue point with label $r_ib_i$,  remark (3) following Definition~\ref{defredblue} implies that $q$ and $\tau(q)$  are incident to the puncture in $\mathcal{T}$. Thus if $\gamma$ is plain then $\gamma\in\mathcal{T}$ and if $\gamma$ is notched then $\tau(\gamma)\in \mathcal{T}$. In the first case $M_\gamma=0$ and in the second case $M_\gamma$ is projective, so in both cases we get a contradiction.
This completes the proof of statement (1).
\smallskip

(2) Let $\Tt$ be a triangulation of type II and assume that we are in the case studied in Figure \ref{figura2}. Let $M_\gamma$ be a non projective syzygy. If $X_\gamma \in \Xx_a$ then, by Lemma \ref{lemadelta}, $M_\gamma \in \underline{\CMP}_\Delta$. If $X_\gamma \notin \Xx_a$ then, by Lemma \ref{lema2y3}, $X_\gamma$ does not lie in the shaded area $A$ in Figure \ref{figura2}. Thus, $X_\gamma \in \{ \tau^{-1} T_b, T_b, \tau T_b, \tau^2 T_c, \tau T_c, T_c\}$ so the  module $M_\gamma $ is injective or projective, this is impossible by Remark \ref{rema util}(b).  The remaining two triangulations of type II can be analyzed in the same manner. 

(3) Let $\Tt$ be a triangulation of type III. Let $M_\gamma$ be a non-projective syzygy. If $X_\gamma \in \Xx_c \cup \Xx_d$ then, by Lemma \ref{lemadelta}, $M_\gamma \in \underline{\CMP}_\Delta$. Suppose that $X_\gamma \notin \Xx_c \cup \Xx_d$ then, by Lemma \ref{lema2y3}, $X_\gamma \notin A_1 \cup A_2 \cup A_3$, see Figure \ref{figura3}. The only positions for $X_\gamma$ under the assumption are those labeled $\clubsuit$ in said figure. The remaining objects correspond to arcs $a$ and $b$ in $\Tt$ or summands $T_a$ and $T_b$ of the cluster tilting object $T$, and the corresponding modules are  zero in $\underline{\CMP}(\B)$. Therefore $M_\gamma \in \underline{\CMP}_\clubsuit \cup \underline{\CMP}_\Delta$.
\end{proof}

\begin{figure}
\centering
\def\svgwidth{5.6in}
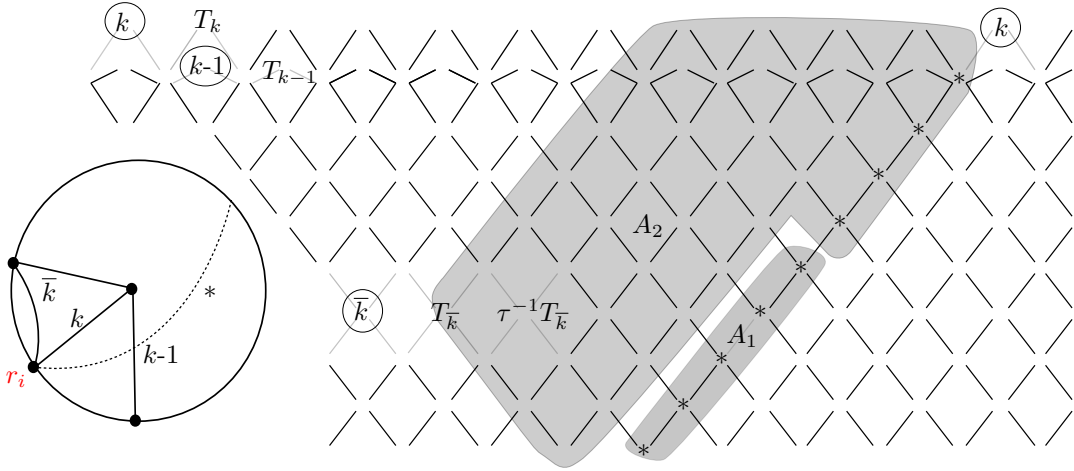
\caption{Proof of Theorem \ref{teofull}(1).}
\label{label1}
\end{figure}

\subsubsection{Auslander-Reiten quiver}
The $\underline{\CMP}$ category of cluster tilted-algebras of type $\mathbb{D}$ has been studied in an abstract way in \cite{CGL}. Our approach is different, we find the modules and their projective resolutions explicitly. Moreover, we can compute the AR quiver of the triangulated category $\underline{\CMP}$ in terms of (relative) elementary moves, analogous to the elementary moves defined in \cite{S}.
\\

According to \cite[Theorem 4.9]{CGL}, the stable category $\underline{\CMP}(\B)$ is equivalent to a union of stable categories $\underline{\mod}S$ where $S$ is a selfinjective cluster-tilted algebra of type $\mathbb{D}$ or $\mathbb{A}_3$. In these cases, the only stable categories having non-trivial morphisms are those arising from cluster-tilted algebras of type $\mathbb{D}_n$, where $n\geqslant 4$. In our setting, the objects $M \in\ind\, \underline{\CMP}(\B)$ such that $\Hom_{\underline{\CMP}}(M,-)$ is not trivial are the modules in $\underline{\CMP}_\odot$. Recall that $\underline{\CMP}_\odot = \{ M(r_i,b_j) \colon j\in \{ i+2,\ldots,i-1\} \}$ and the notation refers to color-labeled endpoints, red $r_i$ and blue $b_j$. %

\begin{df} A \emph{relative elementary move} is a composition of elementary moves $\gamma_0 \raw \gamma_1 \raw \ldots \raw\gamma_{l-1} \raw \gamma_l$ such that
\begin{enumerate}
\item  the arcs $\gamma_0$ and $\gamma_l$ are associated to modules in $\underline{\CMP}_\odot$,
\item the arcs $\gamma_{i}$, $i\in {1,\ldots,l-1}$ are not associated to modules in $\underline{\CMP}_\odot$ and they are are not arcs at the puncture,
\item one of the endpoints of $\gamma_0$ is fixed under all the elementary moves.

\end{enumerate}
\end{df}
In the language of cluster-tilting categories, a relative elementary move is a composition of irreducible morphisms in a sectional path of the AR quiver, such that only the first and the last object are associated to modules in $\underline{\CMP}_\odot$. There are two types of relative elementary moves, defined as follows and described in Figure \ref{relative}.

\begin{itemize}
\item \emph{Red elementary move:} Start with an arc $\gamma(r_i,b_j)$ associated to a module $M(r_i,b_j) \in \underline{\CMP}_\odot$ and make elementary moves in $\tau^{-1}$ direction, keeping the vertex labeled by $b_j$ fixed, until the arc $\gamma^*(r_{i-1},b_j)$ associated to the module $M(r_{i-1},b_j)\in \underline{\CMP}_\odot$ is reached.
\item \emph{Blue elementary move:} Start with an arc $\gamma^*(r_i,b_j)$ associated to a module $M(r_i,b_j) \in \underline{\CMP}_\odot$ and make elementary moves in $\tau^{-1}$ direction, keeping the vertex labeled by $r_i$ fixed, until the arc $\gamma(r_{i},b_{j-1})$ associated to the module $M(r_{i},b_{j-1})\in \underline{\CMP}_\odot$ is reached.
\end{itemize}

\begin{figure}
\centering
\def\svgwidth{5.7in}
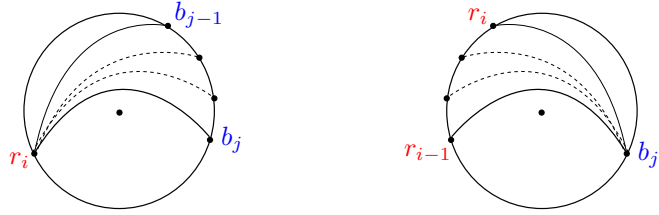
\caption{Red and blue elementary moves in $\underline{\CMP}_\odot$.}
\label{relative}
\end{figure}
The relative elementary moves represent  non-zero morphisms in $\mod \B$. Moreover, such morphisms are irreducible in the category $\underline{\CMP}_\odot$ and this fact leads to a nice description of the category in geometric terms.

\begin{prop} Let $\Tt$ be a triangulation of type I. The AR quiver of the category $\underline{\CMP}_\odot$ can be described as follows:
\begin{enumerate}
\item There is a bijection between relative elementary moves and irreducible morphisms.
\item The AR translation is given by $\widetilde{\tau} (\gamma(r_i,b_j))= \gamma(r_{i+1},b_{ j+1})$.
\end{enumerate}
\end{prop}

\begin{proof}
The correspondence between arcs and indecomposable modules in the set $\underline{\CMP}_\odot$ was studied in Section \ref{typeI}. 
Let us prove that the relative elementary moves are associated to irreducible morphisms. We will study the case of red elementary moves, the prove for blue elementary moves is similar. Let $T$ be the cluster tilting object associated to the triangulation $\Tt$. Assume that the red elementary move is a composition of (at least) two elementary moves $\gamma(r_i,b_j) \raw \gamma_1 \raw \ldots \raw \gamma_{l-1}\raw \gamma(r_{i-1},b_j)$. According to Definition \ref{defredblue} this occurs in two cases:
\\
(1) The vertices labeled $r_{i-1}$ and $r_i$ lie in a last triangle $\Delta=[k,\overline{k},k+1]$ of a set $\Ee_t$. The situation is shown in the left picture in Figure \ref{red}, where the blue dotted line shows the allowed positions for the label $b_j$ with $j \in \{i+2, \ldots, i-1\}$. The objects $X_1, \ldots, X_{l-1}$ are not summands of $\tau T \oplus T$ since the arcs $\tau (\gamma_1), \ldots, \tau(\gamma_{l-1})$ cross the arc $k+2$, and the arcs $ \gamma_1, \ldots, \gamma_{l-1}$ cross the arc $\overline{k}$. 
\\
(2) The vertices labeled $r_{i-1}$ and $r_i$ lie in an internal triangle $\Delta=[k,\overline{k},k+1]$ of a set $\Ee_t$ and such triangle is not the last in $\tau$ order. The case is shown in the right picture in Figure \ref{red}, where the blue dotted line shows the allowed positions for the label $b_j$ with $j \in \{i+2, \ldots, i-1\}$. The objects $X_1, \ldots, X_{l-1}$ are not summands of $\tau T \oplus T$ since the arcs $\tau (\gamma_1), \ldots, \tau(\gamma_{l-1})$ cross the arc $k+1$, and the arcs $ \gamma_1, \ldots, \gamma_{l-1}$ cross the arc $\overline{k}$.  
\\
This implies that the irreducible modules appearing in the composition defined by the red elementary move are non-zero and non-projective. Thus, the composition is a non-zero morphism that is part of a sectional path in the AR quiver of $\underline{\mod}\B$. Also, only the first modules and the last module of the sectional path belong to $\underline{\CMP}_\odot$. It follows that the composition is an irreducible morphism in $\underline{\CMP}_\odot$. Finally, if the red elementary move is already a single elementary move, the corresponding morphism is irreducible in $\underline{\CMP}_\odot$ trivially.
\\
The equivalence of categories $\underline{\mod} S_{(m+d)(m+d-2)} \simeq \underline{\CMP}_\odot$ establishes the structure of the AR quiver. An arc $\gamma(r_i,b_j)$ is a source of one or two relative elementary moves. When $\gamma(r_i,b_j)$ is a source of two moves, that is when $j\notin \{ i+1,i+2\}$ and  $i\notin \{ j+1,j+2\}$, there is a red elementary move $\gamma(r_i,b_j)\raw M(r_{i-1}, b_j)$ and a blue elementary move $\gamma(r_i,b_{j}) \raw \gamma(r_i,b_{j-1})$. Analogously, the arc $\gamma(r_{i-1},b_{j-1})$ is a sink of two elementary moves $\gamma(r_{i-1},b_j) \raw \gamma(r_{i-1},b_{j-1})$ and $\gamma(r_i,b_{j-1})\raw \gamma(r_{i-1},b_{j-1})$. The knitting algorithm produced by the geometric moves generates the AR quiver and the translation is given by $\widetilde{\tau} (\gamma(r_{i},b_j))=\gamma(r_{i+1},b_{ j+1})$.  
\begin{figure}
\centering
\def\svgwidth{5in}
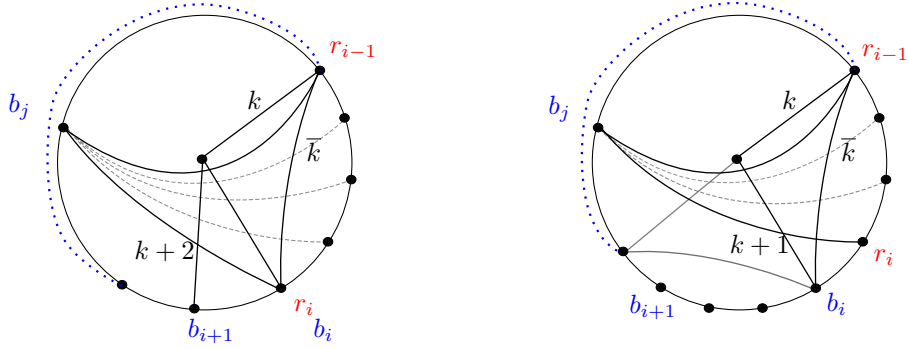
\caption{Possible red elementary moves}
\label{red}
\end{figure}
\end{proof}

The result is consistent with the  following homological argument. Recall that the category $\underline{\CMP}(\B)$ is triangulated and the shift $[1]$ is given by the formal inverse of the syzygy operator $(\Omega)^{-1}$. Also, from \cite[Proposition I.2.3]{RV}, is known that the category has Serre functor $\mathbb{S}$ which is related with the AR translation $\widetilde{\tau}$ via $\mathbb{S}=\widetilde{\tau}[1]$. According to  \cite[Section 3.3]{KR} the category $\underline{\CMP}(\B)$ is 3-CY, this means that $\mathbb{S}=(\Omega)^{-3}$. Combining both results, we have  $\mathbb{S} = (\Omega)^{-3} =\widetilde{\tau} (\Omega)^{-1}$, so $\widetilde{\tau} = (\Omega)^{-2}$. In particular, for each $M(r_i,b_j) \in \underline{\CMP}_\odot$ , we can use Theorem \ref{unlema} to obtain $\widetilde{\tau} M(r_{i},b_j)=M(r_{i+1},b_{ j+1})$. 

\begin{ex} Consider the algebra defined by the triangulation $\Tt$ of type I in the upper part of Figure \ref{arquiver}. The AR quiver of $\underline{\CMP}(\B)=\underline{\CMP}_\odot$ is given by the red and blue elementary moves below. Labeling the arcs in $\Tt$ and computing the intersection number is easy to reconstruct the associated module representations.

\begin{figure}
\centering
\def\svgwidth{6.4in}
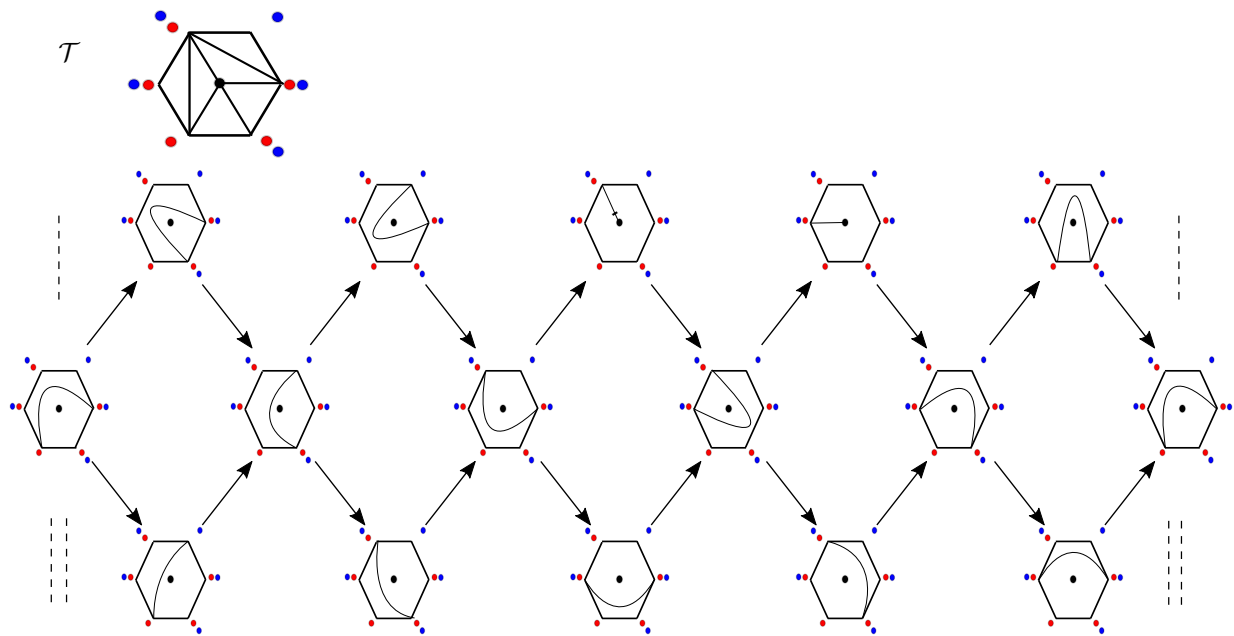
\caption{Geometric realization of the AR quiver of $\underline{\CMP}_\odot$.}
\label{arquiver}
\end{figure}

\end{ex}

%\footnotesize

{} 

%\subsection{Figures for the proof of Theorem \ref{unlema}}

\begin{figure}[b]
\centering
\def\svgwidth{4.8in}
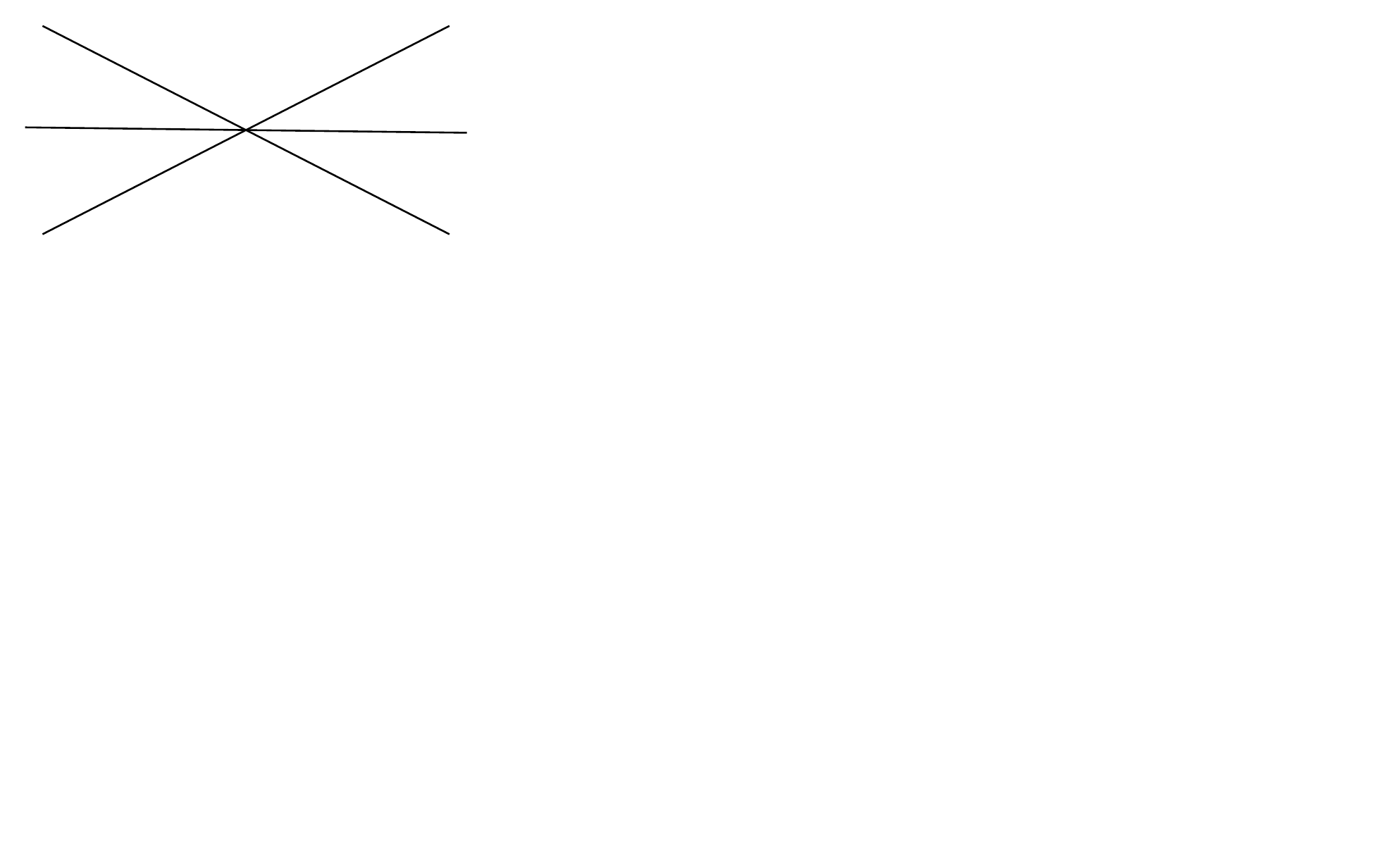
\caption{Proof of Theorem \ref{unlema}, cases (1a), (1b) and (1c). The left column shows part of the triangulation, the middle column, the module $M(r_i,b_j)$ and the right column shows $\Omega M(r_i,b_j)=M(r_{j-1},b_i)$.}
\label{figu1}
\end{figure}
\begin{figure}
\centering
\def\svgwidth{4.8in}
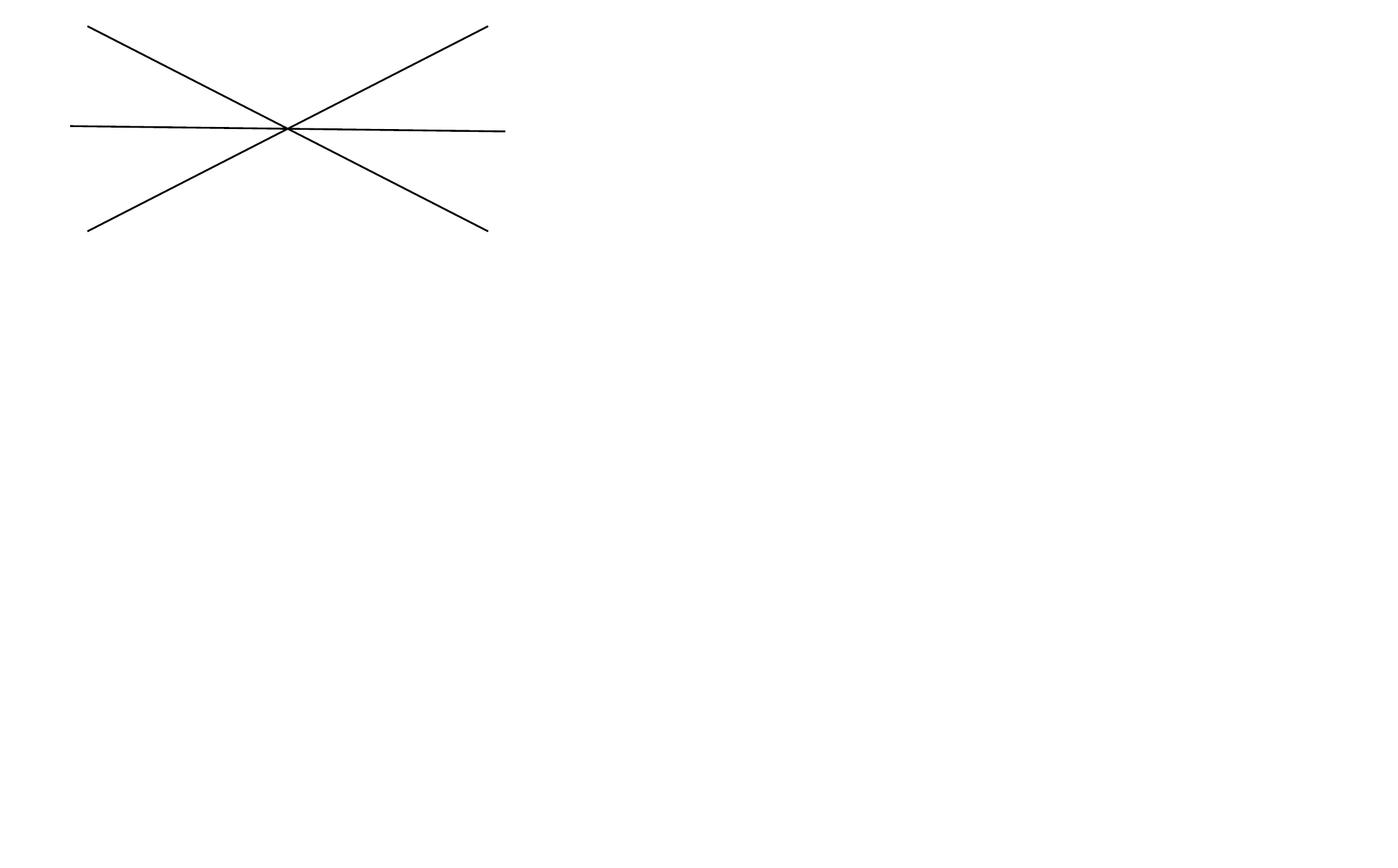
\caption{Proof of Theorem \ref{unlema}, cases (2a), (2b) and (2c). The left column shows part of the triangulation, the middle column, the module $M(r_i,b_j)$ and the right column shows $\Omega M(r_i,b_j)=M(r_{j-1},b_i)$.}
\label{figu2}
\end{figure}
\begin{figure}\label{fig3}
\centering
\def\svgwidth{4.8in}
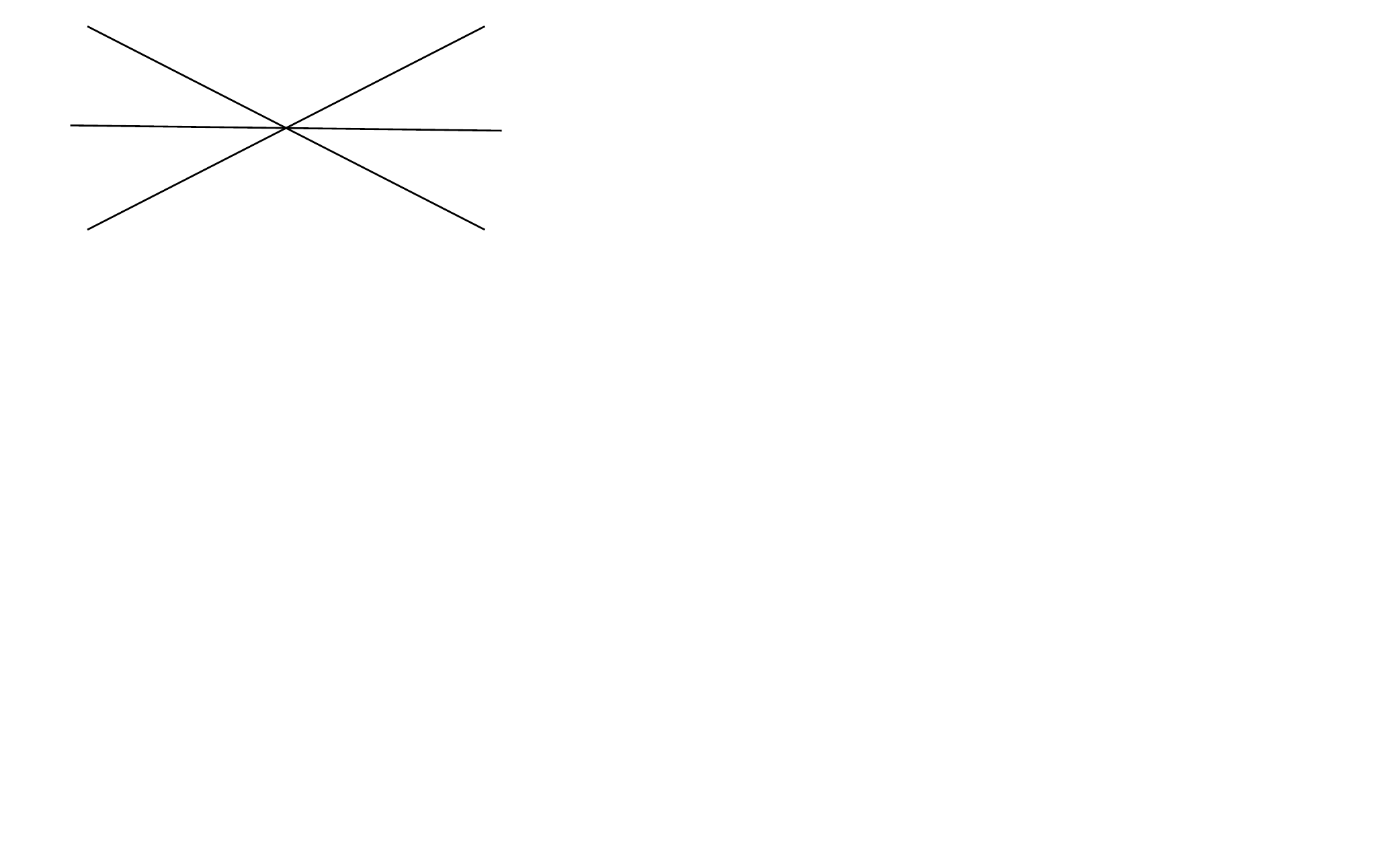
\caption{Proof of Theorem \ref{unlema}, cases (3a), (3b) and (3c). The left column shows part of the triangulation, the middle column, the module $M(r_i,b_j)$ and the right column shows $\Omega M(r_i,b_j)=M(r_{j-1},b_i)$.}
\label{figu3}
\end{figure}
\begin{figure}\label{fig3}
\centering
\def\svgwidth{4.8in}
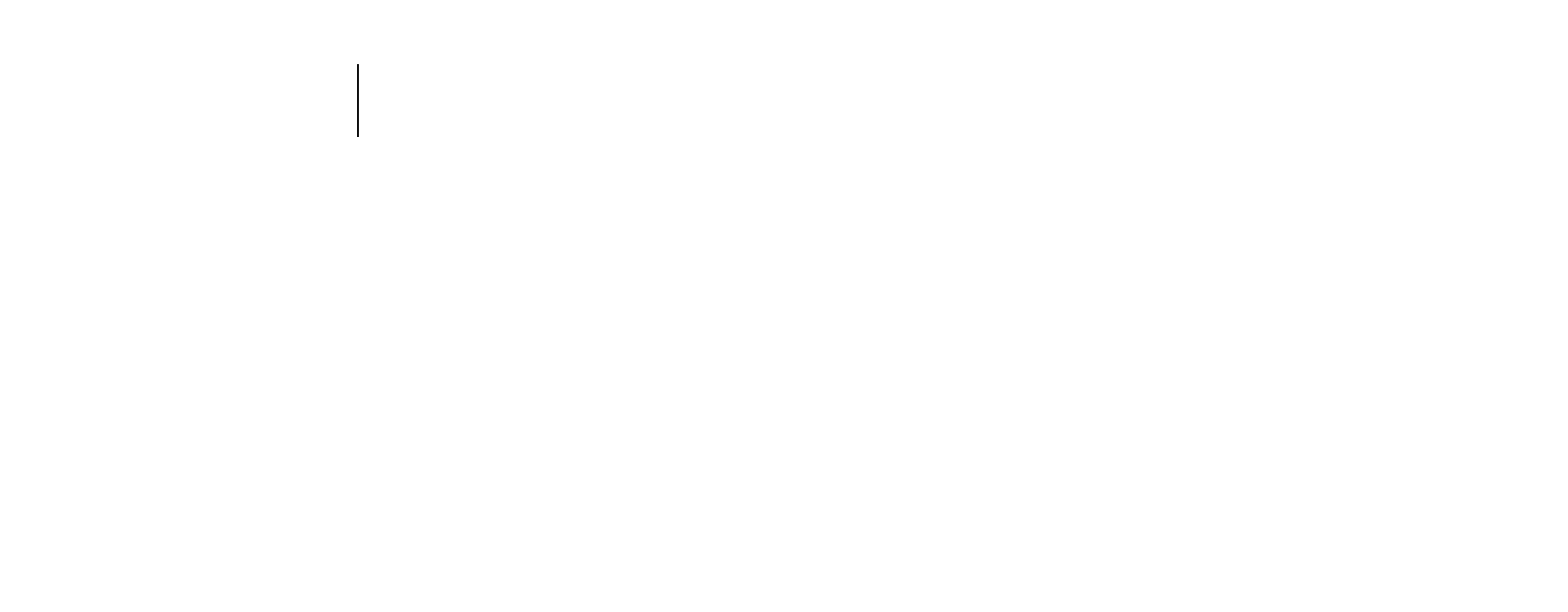
\caption{Proof of Theorem \ref{unlema}, cases (d), (e) and (f).}
\label{figu3}
\end{figure}

%\bibliographystyle{alpha} 
%\bibliography{refe} 

\begin{thebibliography}{}

\bibitem{Am} C. Amiot, Cluster categories for algebras of global dimension 2 and quivers with potential, 
\emph{Ann. Inst. Fourier} {\bf 59} no 6, (2009), 2525--2590. %{\tt arXiv:0805.1035}.

\bibitem{Am11} C. Amiot, On generalized cluster categories, \emph{Representations of Algebras and Related Topics},  EMS Ser. Congr. Rep., Eur. Math. Soc., Z\"urich, 16-02, (2011), 1--53.

\bibitem{ABCP} I. Assem, T. Br\"ustle, G. Charbonneau-Jodoin and
 P.G. Plamondon, Gentle algebras arising from surface triangulations,
\emph{Algebr. Number Theory} {\bf 4}, (2010), no. 2, 201--229.%{\tt arXiv:0903.3347}



\bibitem{ASS} I. Assem, D. Simson  and A. Skowro\'nski, Elements of the Representation Theory of Associative Algebras, \textit{London Math. Soc. Student Texts} \textbf{65} (2006), Cambridge University Press.

\bibitem{ABr} M. Auslander and M. Bridger, Stable module theory. 
{\it Mem. Amer. Math. Soc.\/}, {\bf 94}, (1969).


\bibitem{ARS} M. Auslander, I. Reiten and S. Smal\o, Representation Theory of Algebras,
\emph{Cambridge studies in advanced mathematics}, {\bf 36}, Cambridge University Press, Cambridge, 1997.



%\bibitem{Be05} A. Beligiannis, Cohen--Macaulay modules, (co)torsion pairs and virtually Gorenstein algebras,
%\emph{J. Algebra\/} {\bf 288}, (2005), no. 1, 137--211.

\bibitem{Be00} A. Beligiannis, The homological theory of contravariantly finite subcategories: Auslander-Buchweitz contexts Gorenstein categories and (co-)stabilization,
\emph{Comm. Algebra} {\bf28}, (2000), no. 10, 4547--4596.



\bibitem{BeR} A. Beligiannis and I. Reiten, Homological and homotopical aspects of torsion theories,
\emph{Mem. Amer. Math. Soc.} {\bf 188}, (2007).

\bibitem{BO} P. Bergh and S. Oppermann, Cluster tilting and complexity,
\emph{Int. Math. Res. Not.} {\bf 2011} (2011), no. 22, 5241--5258.


\bibitem{BZ} T. Br\"ustle, J. Zhang, On the cluster category of a marked surface without punctures,
\emph{Algebr. Number Theory} {\bf 5}, (2011), no. 4, 529--566.

\bibitem{BMRRT} A. Buan, R. Marsh, M. Reineke, I. Reiten and G. Todorov, Tilting theory and cluster combinatorics,
\emph{Adv. Math.} {\bf 204}, (2006), no. 2, 572--518.

\bibitem{BMR} A. Buan, R. Marsh and I. Reiten, Cluster-tilted algebras,
\emph{Trans. Amer. Math. Soc.} {\bf 359}, (2007), no. 1, 323--332.



\bibitem{Bu} R. O. Buchweitz, Maximal Cohen-Macaulay modules and Tate-cohomology over Gorenstein rings, \emph{University of Hannover}, preprint (1986).


\bibitem{CCS} P. Caldero, F. Chapoton and R. Schiffler, Quivers with relations arising from clusters ($A_n$ case),
\emph{Trans. Amer. Math. Soc.} {\bf 358}, (2006), no. 3, 1347--1364. 

\bibitem{CSchr} I. Canakci and S. Schroll, Extensions in Jacobian Algebras and Cluster Categories of Marked Surfaces, preprint, {\tt arXiv:1408.2074.}
\bibitem{CGL} X. Chen, S. Geng and M. Lu, The singularity categories of the cluster-tilted algebras of Dynkin type, \emph{Algebr. Represent. Theory} {\bf 18}, (2015), no. 2, 531--554.


\bibitem{DWZ} H. Derksen, J. Weyman and A. Zelevinsky, Quivers with potentials and their representations I: Mutations, \emph{Sel. Math.} {\bf 14}, (2008), no. 1, 59--119.

\bibitem{EJ} E. Enochs and O. Jenda, Relative homological algebra, \emph{de Gruyter Expositions in Mathematics}, {\bf 30}, Walter de Gruyter,  Berlin, 2000.

\bibitem{FZ} S. Fomin and A. Zelevinsky, Cluster algebras I: Foundations, \emph{J. Amer. Math. Soc.} {\bf 15}, (2002), no. 2, 497--529.


%\bibitem{GLS} C. Gei{\ss}, B. Leclerc, and J. Schr{\"o}er, Rigid modules over preprojective algebras, \emph{Invent. Math.} {\bf 165}, (2006), no. 3, 589--632.

\bibitem{GT} A. Gonz{\'a}lez Chaio and S. Trepode, Representation dimension of cluster-concealed algebras,
\emph{Algebr. Represent. Theory} {\bf 16}, (2013), no. 4, 1001--1015.


%\bibitem{H91} D. Happel, On Gorenstein Algebras,  Representation theory of finite groups and finite-dimensional algebras (Bielefeld, 1991), 389--404, 
%Progr. Math., {\bf 95}, Birkh\"auser, Basel, 1991. 


\bibitem{HL} F. Huard and M. Lanzilotta, Self-injective right artinian rings and Igusa Todorov functions,
\emph{Algebr. Represent. Theory} {\bf 16}, (2013), no. 3, 765--770.
 

\bibitem{IT} K. Igusa and G. Todorov, 
On the finitistic global dimension conjecture for Artin algebras,
\emph{Representations of Algebras and Related Topics, AMS, Fields Inst. Comm.\/} {\bf 45}, (2005), 201--204.


%\bibitem{Iy} O. Iyama, Higher-dimensional Auslander--Reiten theory on maximal orthogonal subcategories, \emph{Adv. Math.}  {\bf 210}, (2007), no. 1, 22--50.


\bibitem{Ka} M. Kalck, Singularity categories of gentle algebras,
\emph{Bull. Lond. Math. Soc.} {\bf 47}, (2015), no. 1, 65--74.

 
\bibitem{KR} B. Keller and I. Reiten, Cluster-tilted algebras are Gorenstein and stably Calabi--Yau,
\emph{Adv. Math.} {\bf 211}, (2007), no. 1, 123--151.




\bibitem{L} D. Labardini-Fragoso, Quivers with potentials associated to triangulated surfaces,
\emph{Proc. Lond. Math. Soc.\/} {\bf 98}, (2009), no. 3, 797--839.



\bibitem{Lad} S. Ladkani, Algebras of quasi-quaternion type,
preprint, {\tt arXiv:1404.6834}.



\bibitem{Ma} G. Mata, Funciones de Igusa-Todorov, Ph.D. Thesis, Universidad de la Rep\'ublica (Uruguay).

\bibitem{QZ} Y. Qui and Y. Zhou, Cluster categories for marked surfaces: punctured case, 
preprint, {\tt arXiv:1311.0010}.

\bibitem{R} C. M. Ringel, On the representation dimension of Artin algebras,
\emph{Bull. Inst. Math. Acad. Sin.\/} (N.S.) {\bf 7}, (2012), no. 1, 33--70. 

\bibitem{RV} I. Reiten, and M. Van den Bergh, Noetherian hereditary abelian categories satisfying Serre duality,
\emph{J. Amer. Math. Soc.} {\bf 15}, (2002), no. 2, 295--366.


\bibitem{S} R. Schiffler, A geometric model for cluster categories of type $D_n$, \emph{J. Alg. Comb.\/} {\bf 27}, no. 1, (2008) 1--21.


\bibitem{S14} R. Schiffler, Quiver representations, \emph{CMS Books in mathematics}, Springer Verlag, (2014).




\bibitem{VD} Y. Valdivieso-D{\'\i}az, Jacobian algebras with periodic module category and exponential growth, {\em J. Algebra\/} {\bf 449} (2016), 163--174.% {\tt arXiv:1309.2708}.




%\bibitem{Zo} M. del Zotto, More Arnold's $\mathcal{N}$= 2 superconformal gauge theories, \emph{J. High Energy Physics}, {\bf 2011}, (2011), no. 11, 1--25.

%\bibitem[S]{S2} R. Schiffler, A cluster expansion formula ($A_n$
  %case), Electron. J. Combin. 15 (2008), \#R64 1.% {\tt arXiv:math.RT/0611956}.

%\bibitem[S2]{S3} R. Schiffler, On cluster algebras arising from
%unpunctured surfaces II, 
%\emph{Adv. Math.} {\bf 223}, (2010), 1885--1923.  % {\tt arXiv:0809.2593}.

%\bibitem[ST]{ST} R. Schiffler and H. Thomas, On cluster algebras arising from unpunctured surfaces, \emph{Int. Math. Res. Not.\/} no. 17, (2009), 3160--3189.








\end{thebibliography}

\end{document}